\documentclass[14pt]{article}
\usepackage{amsmath}
\usepackage{amssymb}
\usepackage{mathrsfs}
\usepackage{amscd}
\usepackage{amsthm,color}
\usepackage[matrix,frame,import,curve]{xypic}
\usepackage{authblk}
\usepackage{tikz}
\usepackage{dirtytalk}

\let\margin\marginpar
\newcommand\myMargin[1]{\margin{\raggedright\scriptsize #1}}
\renewcommand{\marginpar}[1]{\myMargin{#1}}

\newtheorem{lemma}{Lemma}[section]
\newtheorem{theorem}[lemma]{Theorem}
\newtheorem{corollary}[lemma]{Corollary}

\newtheorem{prop}[lemma]{Proposition}
\theoremstyle{definition}
\newtheorem{definition}[lemma]{Definition}
\newtheorem{example}[lemma]{Example}

\newtheorem{remark}[lemma]{Remark}

\theoremstyle{remark}
\newtheorem*{proof*}{Proof}
\numberwithin{equation}{section}

\def\RHom{{\mathbf{R}\mathrm{Hom}}}

\def\D{\mathrm{D}}

\def\Hom{{\mathrm{Hom}}}

\def\Bimod{{-\mathrm{Bimod}}}

\def\CC{{\mathbb C}}

\def\cG{{\mathcal{G}}}

\def\cO{{\mathcal{O}}}

\def\cL{{\mathcal{L}}}

\def\fm{{\mathfrak{m}}}
\def\fn{{\mathfrak{n}}}
\def\fd{{\mathfrak{d}}}

\def\lg{{\langle}}
\def\rg{{\rangle}}
\def\lgg{{\langle\langle}}
\def\rgg{{\rangle\rangle}}
\def\ol{\overline}
\def\ot{{\otimes}}
\def\cot{{\widehat{\otimes}}}

\def\wh{\widehat}

\def\Lm{{\Lambda}}

\def\p{{\prime}}
\def\pp{{\prime\prime}}
\def\1{{\bf{1}}}

\def\cy{{\mathrm{cyc}}}
\def\Aut{{\mathrm{Aut}}}
\def\id{{\mathrm{Id}}}
\def\der{{\mathrm{Der}}}
\def\cder{{\mathrm{cDer}}}
\def\dder{{\mathrm{\mathbb{D}er}}}
\def\d{{\mathrm{d}}}

\def\im{{\mathrm{im}}}

\def\fD{{\mathfrak{D}}}

\def\per{{\mathrm{per}}}

\def\rJ{{\mathrm{J}}}

\def\hD{\wh{\fD}}

\title{Noncommutative Mather-Yau theorem and its applications to Calabi-Yau algebras}
\date{}
\author[1]{Zheng Hua\thanks{huazheng@maths.hku.hk}}
\author[21]{Gui-Song Zhou \thanks{zhouguisong@nbu.edu.cn}}
\affil[1]{Department of Mathematics, the University of Hong Kong, Hong Kong SAR, China}
\affil[2]{School of Mathematics and Statistics, Ningbo University, Ningbo, China}

\begin{document}
\maketitle
\begin{abstract}
In this article, we prove that for a finite quiver $Q$ the equivalence class of a potential up to formal change of variables of the complete path algebra $\wh{\CC Q}$, is determined by its Jacobi algebra together with the class in its 0-th Hochschild homology represented by the potential assuming the Jacobi algebra is finite dimensional. This is an noncommutative analogue of the famous theorem of Mather and Yau on isolated hypersurface singularities. We also prove that the right equivalence class of a potential is determined by its sufficiently high jet assuming the Jacobi algebra is finite dimensional. These two theorems can be viewed as a first step towards the singularity theory of noncommutative power series. As an application, we show that if the Jacobi algebra is finite dimensional then the corresponding complete Ginzburg dg-algebra, and the (topological) generalized cluster category thereof, are determined by the isomorphic type of the Jacobi algebra together with the class in its 0-th Hochschild homology  represented by the potential. 
\end{abstract}

\section{Introduction}
\newtheorem{maintheorem}{\bf{Theorem}}
\renewcommand{\themaintheorem}{\Alph{maintheorem}}
\newtheorem{mainconjecture}[maintheorem]{\bf{Conjecture}}
\renewcommand{\themainconjecture}{}

Let $f\in \cO_{\CC^n,0}$ be a germ of holomorphic function at $0\in \CC^n$. The \emph{Tyurina algebra} $T_f$ associate to the function $f$ is defined to be the quotient algebra $\cO_{\CC^n,0}/(f, \frac{\partial f}{\partial x_1},\ldots, \frac{\partial f}{\partial x_1})$. In the 80s, Mather and Yau proved that if $f$ has an isolated critical point then the isomorphism type of the hypersurface singularity $\{f=0\}$ is determined by the isomorphism type of $T_f$ (see Section 1 of \cite{MY}). The Mather-Yau theorem is one of the most important results in singularity theory. In the past few decades, various generalizations of it have been obtained. In \cite{GH85}, Gaffney and Hauser proved a vast generalization of the Mather-Yau theorem, where in particular the isolatedness assumption is removed. Recently, Greuel and Pham generalizes Mather-Yau theorem to the case of formal power series \cite{GP}. They also prove a weaker version of the theorem for the positive characteristic case. 

The main result of this paper is a Mather-Yau type theorem in noncommutative geometry. Before stating the theorem, we recall some background of noncommutative geometry in the context of quiver with potential. For a finite quiver $Q$, denote by $\wh{\CC Q}$ its complete path algebra (see definition in Section \ref{sec:Pre}). A \emph{potential} $\Phi$ is an element of $\wh{\CC Q}$ modulo cyclic permutation of paths, i.e. $\Phi\in \wh{\CC Q}_{\ol{\cy}}:=\wh{\CC Q}/[\wh{\CC Q},\wh{\CC Q}]^{cl}$. The \emph{Jacobi algebra} $\wh{\Lm}(Q,\Phi)$ associate to the pair $(Q,\Phi)$ is defined to be quotient algebra of $\wh{\CC Q}$ modulo the closure of the two sided ideal generated by all the cyclic derivatives of $\Phi$ (see Definition \ref{superpotential-complete}). The natural projection $\wh{\CC Q}\to \wh{\Lm}(Q,\Phi)$ induces a map $\wh{\CC Q}_{\ol{\cy}}\to \wh{\Lm}(Q,\Phi)_{\ol{\cy}}$, which sends $\Phi$ to a class $[\Phi]$. Note that $\wh{\Lm}(Q,\Phi)_{\ol{\cy}}$ is the 0-th  Hochschild homology group of $\wh{\Lm}(Q,\Phi)$ when it is finite dimensional. Two potentials $\Phi$ and $\Psi$ are called \emph{right equivalent} if they differ by a formal change of variables (see Definition \ref{right-equivalent}).

\begin{maintheorem}(Theorem \ref{ncMY})\label{mainthm-MY}
Fix a finite quiver $Q$. Let $\Phi,\Psi\in \wh{\CC Q}_{\ol{\cy}}$ be two potentials of order $\geq 3$, such that the Jacobi algebras $\wh{\Lm}(Q,\Phi)$ and $\wh{\Lm}(Q,\Psi)$ are both finite dimensional. Then the following two statements are equivalent:
\begin{enumerate}
\item[(1)] There is an $\CC Q_0$-algebra isomorphism $\gamma: \wh{\Lm}(Q,\Phi)\cong\wh{\Lm}(Q,\Psi)$ so that $\gamma_*([\Phi])=[\Psi]$, where $\gamma_*: \wh{\Lm}(Q,\Phi)_{\ol{\cy}}\to\wh{\Lm}(Q,\Psi)_{\ol{\cy}}$ is induced by $\gamma$.
\item[(2)] $\Phi$ and $\Psi$ are right equivalent.
\end{enumerate}
\end{maintheorem}
To make an analogue with the classical Mather-Yau theorem, $\Phi$ can be viewed as a formal function on the noncommutative space associate to $Q$. The Jacobi algebra is the noncommutative analogue of the Milnor algebra $\cO_{\CC^n,0}/(\frac{\partial f}{\partial x_1},\ldots, \frac{\partial f}{\partial x_1})$. The finite dimensional condition on $\wh{\Lm}(Q, \Phi)$ can be interpreted as a noncommutative version of isolatedness of singularity.  A basic example will be the quiver with one node and $n$ loops. In this case, the complete path algebra is simply the complete free algebra of $n$ generators and a potential is a formal series of necklaces. 

In recent years, the study of singularity theory of noncommutative functions has attracted more and more attention due to the motivations from algebraic geometry and representation theory. In his milestone paper \cite{Ginz}, Ginzburg defined the notion of Calabi-Yau algebra (see definitions in Section \ref{sec:CY-Ginzdg}) and provided a construction using quiver with potential. Ginzburg conjectured that all 3-dimensional Calabi-Yau algebras should arise as Jacobi algebras. This conjecture has been confirmed by Van den Bergh for all complete Calabi-Yau algebras \cite{VdB15}, while the general case has been disproved by Ben Davison \cite{Da}.

In \cite{Ginz}, Ginzburg only considered Calabi-Yau algebras that are homologically smooth, whereas Jacobi algebras are not necessarily smooth. Indeed it is more natural to define Calabi-Yau algebras under the framework of dg algebras. For every finite quiver $Q$ and potential $\Phi$, Ginzburg defined a natural complete dg-algebra $\hD(Q,\Phi)$ associating to it, where  $H^0(\hD(Q,\Phi))\cong \wh{\Lm}(Q,\Phi)$ (see Definition \ref{Ginzalg}). Keller proved that $\hD(Q,\Phi)$ always satisfies the (smooth) Calabi-Yau property, regardless of whether $\wh{\Lm}(Q,\Phi)$ is smooth or not \cite{KV09}. A geometric proof for this statement was given by Van den Bergh in the Appendix of \cite{KV09}. The result of Keller and Van den Bergh suggests that we should view $\hD(Q,\Phi)$ as a ``derived thickening" of the the Jacobi algebra $\wh{\Lm}(Q,\Phi)$. This point of view, already contained in Ginzburg's original paper, is motivated by Donaldson-Thomas theory. In Donaldson-Thomas theory, the moduli space of stable sheaves on Calabi-Yau threefold is locally isomorphic to a critical set of a function. Indeed, the moduli space admits an enhancement as a derived scheme (stack) of virtual dimension zero in the sense of \cite{BCFHR}. Although the moduli space itself is usually highly singular, the derived moduli space is a better behaved object. It is well-known that the derived structure is determined by the moduli space together with its obstruction theory.

If we make a heuristic comparison between the Jacobi algebra and the moduli space of stable sheaves in Donaldson-Thomas theory, then it is natural to ask the following question:{\bf{ to what extent the complete Ginzburg dg-algebra is determined by the algebra structure on the Jacobi algebra?}}  In this paper, we give an answer to this question using Theorem \ref{mainthm-MY}.

\begin{maintheorem}(Theorem \ref{rigidity-Ginzburg})\label{mainthm-rigidity}
Fix a finite quiver $Q$. Let $\Phi,\Psi\in \wh{\CC Q}_{\ol{\cy}}$ be two potentials of order  $\geq 3$, such that the Jacobi algebras
$\wh{\Lm}(Q,\Phi)$ and $\wh{\Lm}(Q,\Psi)$ are both finite dimensional. Assume there is an $\CC Q_0$-algebra isomorphism $\gamma: \wh{\Lm}(Q,\Phi)\to\wh{\Lm}(Q,\Psi)$  so that
$\gamma_*([\Phi])=[\Psi]$.  Then there exists a dg-$\CC Q_0$-algebra isomorphism
\[
\xymatrix{
\Gamma: \hD(Q,\Phi)\ar[r]^{\cong} &\hD(Q,\Psi)
}
\]
such that $\Gamma(t_i)=t_i$ for all nodes $i$ of $Q$. 
\end{maintheorem}
As an immediate corollary, we show that the (topological) \emph{generalized cluster category} is determined by the Jacobi algebra together with the class in its 0-th Hochschild homology represented by the potential (Corollary \ref{clustercat}). This result has some interesting application in algebraic geometry. In \cite{HuaKeller}, the first author and Keller use it to show that the singularity underlying a three dimensional flopping contraction is determined by the noncommutative deformation of the exceptional curve.

Now we briefly summarize the ideas used in the proof of Theorem \ref{mainthm-MY}.
In a series of classical papers \cite{Ma68}\cite{Ma69}, Mather studied various equivalence relations on germ of smooth and analytic functions. He has observed that the tangent spaces of these equivalence classes can be interpreted as Jacobi ideal and its variations. Moreover, Mather has found infinitesimal criteria of checking whether two germs lie in the same equivalence class. The classical Mather-Yau theorem can be proved using it.

In noncommutative differential calculus, the notion of cyclic derivative was first defined by Rota, Sagan and Stein \cite{RRS}.  It plays an essential role in the construction of Ginzburg dg algebra.
The main difficulty to implement Mather's technique on noncommutative functions is that the tangent space of a right equivalence class of functions, which is a subspace of the push forward of the space of derivations, admits no module structure. So  in noncommutative world, there is no a priori relationship between the tangent space of right equivalence class and the Jacobi ideal. We overcome this problem by observing that both space of derivations and space of cyclic derivations are quotients of the space of double derivations, which does admit a bimodule structure. Moreover, the actions of derivations and cyclic derivations on the space of potentials have the identical orbits. Another difficulty lies in the fact that the group of formal automorphisms is infinite dimensional. To reduce it to a finite dimensional problem, we establish a finite determinacy result for noncommutative formal series, which is of independent interest.
\begin{maintheorem}(Theorem 3.16)
Let $Q$ be a finite quiver and  $\Phi\in \wh{\CC Q}_{\ol{\cy}}$. If the Jacobi algebra $\wh{\Lm}(Q,\Phi)$ is finite dimensional then $\Phi$ is finitely determined (see definition in Section \ref{sec:fd}). More precisely, if
$\wh{\rJ}(Q,\Phi) \supseteq \wh{\fm}^r$ for some integer $r\geq 0$ then $\Phi$ is $(r+1)$-determined. In particular, any potential with finite dimensional Jacobi algebra is right equivalent to a potential in $\CC Q_{\cy}$.
\end{maintheorem}

 The third difficulty is that in the formal world not all derivations come from tangents of automorphisms. To be more precise, we are not allowed to take translation of origin as in the smooth or analytic world. We resolve this problem by proving a bootstrapping lemma on Jacobi ideals (see Proposition \ref{Bootstrapping}).

The paper is organized as follows. In Section \ref{sec:Pre}, we collect basic notations and results on noncommutative differential calculus. In particular, the definition of Jacobi algebra for path algebra (resp. complete path algebra) are presented in Definition \ref{superpotential-free} (resp. in Definition \ref{superpotential-complete}). Section \ref{sec:Proof} is devoted to the proof of Theorem \ref{mainthm-MY}.  In Section \ref{sec:CY-Ginzdg}, we recall the definitions of complete Ginzburg dg-algebras and Calabi-Yau algebras, and prove Theorem \ref{mainthm-rigidity}.

\paragraph{Acknowledgments.} We are benefit from the valuable communication with Bernhard Keller, Gert-Martin Greuel, Martin Kalck and Sasha Polishchuk.  The research of the first author was supported by RGC General Research Fund no. 17330316,  no. 17308017 and Early Career grant no. 27300214. The research of the second author was supported by NSFC grant no. 11601480 and RGC Early Career grant no. 27300214.

\section{Preliminaries}\label{sec:Pre}

In this section, we collect basic notations and terminologies that are of concern. In particular,  we recall the definition of  Jacobi algebras of path algebras and of complete path algebras of finte quivers. In addition, for completeness and reader's convenience, some relevant well-known facts on complete path algebras are presented in full details in our own notations.
Throughout, we fix a base commutative ring $k$ with unit and  a finite dimensional separable $k$-algebra $l$. Unadorned tensor products are over $k$.

\subsection{Basic terminologies}
A $l$-algebra $A$ is a $k$-algebra $A$ equipped with a $k$-algebra monomorphism $\eta: l\to A$. Note that the image of $l$ is in general not central even when $l$ is commutative. We are mainly interested in the case when $l=k e_1+\ldots +k e_n$ for central orthogonal idempotents $e_i$.

%Let $U$ be a right $l$-module and $V$ be  left $l$-module. The tensor product $U\ot_l V$ is the quotient of $U\ot V$ by the $k$-submodule generated $xe_i\ot y-x\ot e_i y$ for any $x\in U$, $y\in V$ and $i=1,\ldots, n$. If $U$ and $V$ are $l$-bimodules, we call a map $f:M\to N$ \emph{$l$-linear} if it is a morphism of $l$-bimodules. 

Let $A$ be an associative ($l$-)algebra. Set $A_{ij}:=e_i A e_j$.  The multiplication map $\mu:A\ot A\to A$ is equivalent with its restriction 
\[
\mu: A_{ij}\ot A_{jk}\to A_{ik}~~~~\text{for all } i,j,k=1,\ldots,n.
\]
In other words, it factors through the map $A\ot_l A\to A$, which will be denoted by the same symbol $\mu$ by an abuse of notations.
The tensor product space $A\ot A$ carries two commuting $A$-bimodule structures, called the \emph{outer} (resp. \emph{inner}) bimodule structure, defined by
\[
a(b^\p\ot b^{\p\p})c:=ab^\p\ot b^{\p\p}c~~~~~(~\text{resp.} \quad a*(b^\p\ot b^{\p\p})*c:=b^\p c\ot ab^{\p\p} ~).
\]
We denote by $A\stackrel{out}{\ot}A$ (resp $A\stackrel{in}{\ot} A$) the bimodule with respect to the outer (resp. inner) structure. Because $l$ is a sub algebra of $A$, they are in particular $l$-bimodules.

The flip map $\tau:A\stackrel{out}{\ot}A \to A\stackrel{in}{\ot}A$, defined by $a^\p\ot a^\pp$ to $a^\pp\ot a^\p$,   is an isomorphism of $A$-bimodules, $\mu:A\stackrel{out}{\ot}A \to A$ is a homomorphism of $A$-bimodules but in general $\mu: A\stackrel{in}{\ot}A \to A$ is not. Unless otherwise stated, we simply view $A\ot A$ as the bimodule $A\stackrel{out}{\ot}A$.  Also, the category of $A$-bimodules is denoted by $A\Bimod$. The multiplication map $\mu$ is $l$-linear.

A \emph{(relative) $l$-derivation of $A$ in an  $A$-bimodule $M$} is defined to be a $l$-linear map $\delta:A\to M$  satisfies the Leibniz rule, that is $\delta(ab) =\delta(a)b+ a\delta(b)$ for all $a,b\in A$.  
%If $e$ is an idempotent, it follows that
%\[
%\delta(e)=\delta(e^2)=\delta(e)e+e\delta(e)=2\delta(e^2)=2\delta(e).
%\] 
It follows that $\delta(l)=0$ and $\delta(A_{ij})\subset M_{ij}$ with $M_{ij}:=e_i M e_j$.
Denote  by $\der_l(A,M)$ the set of all $l$-derivations of $A$ in $M$, which naturally carries a $k$-module structure.  The elements of
\[
\der_l ( A): =\der_l(A,A) ~~~~~ (~\text{resp.} \quad \dder_l (A):= \der_l(A,A\ot A)~ )
\]
are called the  \emph{$l$-derivations of $A$} (resp. \emph{double $l$-derivations of $A$}). For a general double derivation $\delta\in \dder_l A$ and $f\in A$, we shall  write in Sweedler's notation that
\begin{align}\label{Sweedler}
\delta(f)= \delta(f)^\p\ot\delta(f)^\pp.
\end{align}
The inner bimodule structure of $A\ot A$ naturally yields a bimodule structure on $\dder_l(A)$.  In contrast, $\der_l (A)$  doesn't  have canonical left nor  right $A$-module structures in general.  The multiplication map $\mu$ induces  a $k$-linear map $\mu\circ-: \dder_l(A)\to \der_l(A)$ given by $\delta\mapsto \mu\circ \delta$. We refer to $\mu\circ \delta$ the $l$-derivation corresponding to the double $l$-derivation $\delta$.

Let us put on the space of $k$-module endomorphisms $\Hom_k(A,A)$  the $A$-bimodule structure  defined by
\[
a_1fa_2: b\mapsto a_1 f(b) a_2, ~~~~~ f\in \Hom_k(A,A), ~ a_1, a_2, b\in A.
\]
Though  the map $\dder_l(A) \xrightarrow{\mu\circ- } \Hom_k(A,A)$ doesn't preserves  bimodule structures,  the map
\[
\mu\circ \tau\circ- : \dder_l (A) \to \Hom_k(A,A)
\]
is clearly a homomorphism of $A$-bimodules. Denote by $\cder_l (A)$ the image of this map and call its elements \emph{cyclic  $l$-derivations of $A$}. For a double $l$-derivation $\delta\in \dder_l(A)$, we shall refer to $\mu\circ \tau\circ \delta$ the cyclic $l$-derivation corresponding to $\delta$. Note that by definition $\cder_l(A)$ is  an $A$-subbimodule of $\Hom_k(A,A)$, and hence is itself an $A$-bimodule.

The $A$-bimodule $\Omega_{A|l}$, of \emph{noncommutative (relative) K\"ahler $l$-differentials of $A$}, is defined to be the kernel of the multiplication map $\mu: A \ot_l A\to A$. The exterior differentiation map
\[
d: A\to \Omega_{A|l}, \quad a\mapsto d a:=1\ot a-a\ot 1
\]
is then a $l$-derivation of $A$ in the bimodule $\Omega_{A|l}$. There is a canonical isomorphism of $k$-modules
\[
\Hom_{A-{\rm Bimod}}(\Omega_{A|l},M) \xrightarrow{\cong} \der_l(A,M), \quad f\mapsto f\circ d.
\]
In another word, the exterior map $d:A\to \Omega_{A|l}$ is a universal $l$-derivation of $A$.

We collect some trivial properties on (cyclic) derivations in the following lemma.
\begin{lemma}\label{cycprop}
Let  $A$ be a $l$-algebra and fix an element $\Phi\in A_\cy:=A/[A,A]$. Let $\pi:A\to A_\cy$ be the canonical projection and $\phi\in A$ a representative of $\Phi$.
\begin{enumerate}
\item[$(1)$]  $\xi([A,A]) \subseteq [A,A]$ for every $\xi\in \der_l(A)$. Consequently, the assignment $\der_l(A) \ni \xi \mapsto \pi(\xi (\phi))$ only depends on $\Phi$ and defines a $k$-linear map   $\Phi_{\#}: \der_l (A) \to A_\cy $.
\item[$(2)$]  $D([A,A])=0$ for every $D\in \cder_l (A)$. Consequently,  the assignment $\cder_l(A)\ni D \mapsto D(\phi)$ only depends on $\Phi$ and defines an $A$-bimodule morphism $\Phi_*: \cder_l (A)\to A$.
\item[$(3)$] We have the following commutative  diagram:
\begin{align}
\xymatrix{
\dder_l (A) \ar@{->>}[r]^-{\mu\circ \tau\circ-} \ar[d]^{\mu\circ-} &   \cder_l( A)\ar[r]^-{\Phi_*} & A\ar[d]^{\pi} \\
\der_l ( A) \ar[rr]^{\Phi_{\#}} & & A_\cy .
}   \label{derivation}
\end{align}
Consequently, if $\dder_l(A) \xrightarrow{\mu\circ-} \der_l(A)$ is surjective then $\im (\Phi_{\#}) = \im (\pi\circ \Phi_*)$.
\end{enumerate}
\end{lemma}
\begin{proof}
Note that $A_\cy$ admits only a $k$-module structure, but not a $l$-module structure.
Part (1) of the lemma and the second statement of  part (2) are clear.  To see the first statement of part (2), let $\delta\in \dder_l(A)$ be a double $l$-derivation and $f,g\in A$. Then
\begin{align*}
(\mu\circ \tau\circ \delta)([f,g])=\delta(f)^\pp g \delta(f)^\p+\delta(g)^\pp f\delta(g)^\p-\delta(g)^\pp f\delta(g)^\p -\delta(f)^\pp g\delta(f)^\p=0.
\end{align*}
From this we know $D([A,A])=0$ for every $D\in \cder_l(A)$. Part (3) of the lemma is a direct consequence of the fact that $(\mu\delta)(f)-(\mu\tau\delta)(f)\in [A,A]$ for every $\delta\in \dder_l(A)$ and $f\in A$.
\end{proof}

%%%%%%%%%%%%%%%%%%%%%%%%%%%%%%%%%%%%%%%%%%%%%%%%%%%%%%%

\subsection{Jacobi algebras of quivers with potentials}
Let $Q$ be an arbitrary finite quiver. Denote by $Q_0$ and $Q_1$ the sets of nodes and arrows of the quiver. Define $s,t: Q_1\to Q_0$ to be the \emph{source} and \emph{target} maps. Denote by $kQ$ the path algebra of $Q$ with respect to the path concatenation defined as follows.  Let $a$ be a path from $i$ to $j$ and $b$ be a path from $j$ to $k$. Then $ab$ is a path from $i$ to $k$ defined by $a$ followed by $b$.  Denote by $e_i$ the empty path at the node $i$. Denote by $Q_1^{(ij)}$ the set of arrows with source $i$ and target $j$.
Denote by $\fm$ the two sided ideal generated by all arrows and by $l$ the subalgebra $kQ_0$ which is canonically isomorphic to the quotient algebra $kQ/\fm$. As a consequence, $kQ$ is an \emph{augmented} $l$-algebra.  An element of $kQ$ is a finite $k$-linear combination of paths. The \emph{length} or \emph{degree} $|p|$ of a path $p$ is defined in the obvious way. The ideal $\fm^r$ consists of finite sum of paths of length $\geq r$. 

%A (right) $kQ$-module is determined by collection of $k$-modules $\{V_i\}_{i\in Q_0}$ and $k$-linear maps $T_a: V_{s(a)}\to V_{t(a)}$ for any $a\in Q_1$.
It is easy to check that $l$-derivations  of $kQ$ in an $kQ$-bimodule  are uniquely determined by their value at all $a\in Q_1$. Thus $kQ$ has double derivations $\frac{\partial}{\partial a}:kQ\to kQ\ot kQ$ given by
\begin{align}\label{Di}
\frac{\partial}{\partial a}( b) = \delta_{a, b} ~e_{s(a)}\ot e_{t(a)}.
\end{align}
The $kQ$-bimodule $\dder_l (kQ)$ is generated by $\frac{\partial}{\partial a}$, $a\in Q_1$, and the map $\mu\circ-: \dder_l (kQ)\to \der_l (kQ)$ is surjective. 
By \eqref{Di}, $\mu\circ \frac{\partial}{\partial a}=0$ for any arrow $a$ such that $s(a)\neq t(a)$. 
 Denote by $D_{a}$ the cyclic $l$-derivation corresponding to  $\frac{\partial~}{\partial a}$, and so it take a path $p$ to
\[
D_{a}(p)=\sum\nolimits_{p=uav}vu
\] where $u,v$ are paths. By (2) of Lemma \ref{cycprop}, $D_a([kQ,kQ])=0$ for any $a\in Q_1$. In particular, if $p$ is a path such that $s(p)\neq t(p)$ then $D_a(p)=0$ for any $a$.

This cyclic derivation in the special case of $n$-loop quiver was first discovered by Rota, Sagan and Stein \cite{RRS}. Clearly, the $kQ$-bimodule $\cder_l (kQ)$ is generated by $D_a$, $a\in Q_1$. In addition, the $kQ$-bimodule $\Omega_{kQ|l}$ is generated by $\d a$, $a\in Q_1$.

\begin{definition}\label{superpotential-free}
Elements of $kQ_\cy:=kQ/[kQ,kQ]$ are called \emph{potentials} of $kQ$. Given a potential $\Phi\in kQ_\cy$, the two sided ideal
\[
\rJ(Q,\Phi):=\im(\Phi_*)
\]
is called the \emph{Jacobi ideal} of  $\Phi$, where $\Phi_*:\cder_l(kQ)\to kQ$ is the $kQ$-bimodule homomorphism constructed in Lemma \ref{cycprop} (2). The associative algebra
\[
\Lm(Q, \Phi) := kQ/\rJ(Q,\Phi)
\]
is called the  \emph{Jacobi algebra} of  $\Phi$. Like $kQ$, it is an augmented $l$-algebra.
\end{definition}

The above  definition of Jacobi algebras coincides with the conventional one
because $\rJ(Q,\Phi)$ is  generated by $\Phi_*(D_a)$, $a\in Q_1$, as an ideal of $kQ$. Jacobi algebras of path algebras of quivers are key objects of interest in literatures.

%%%%%%%%%%%%%%%%%%%%%%%%%%%%%%%%%%%%%%%%%%%%%%%%%%%%%%%

\subsection{Complete Jacobi algebras}

Given a finite quiver $Q$, denote by $\wh{kQ}$ the completion of $kQ$ with respect to the two sided ideal $\fm\subset kQ$.
Elements  of $\wh{kQ}$ are (infinite) formal series $\sum_{w} a_w w$, where $w$ runs over all paths (of finite length)  and $a_w\in k$. Note that $\wh{kQ}$ contains  the $kQ$ as a $l$-subalgebra. Let $\wh{\fm} \subseteq \wh{kQ}$ be the ideal generated by all arrows.  For $r\geq0$,  $\wh{\fm}^r$ consists of formal series with no terms of degree $< r$. For any subspace $U$ of $\wh{kQ}$, let $U^{cl}$ be the closure of $U$ with respect to the $\wh{\fm}$-adic topology on $\wh{kQ}$. Note that $U^{cl}= \cap_{r\geq 0 } (U+\wh{\fm}^r)$.

Clearly, $l$-derivations  of  $\wh{kQ}$  are  uniquely determined by their value at all $a\in Q_1$.  However, it is generally not true for $l$-derivations  of  $\wh{kQ}$ in an arbitrary  $\wh{kQ}$-bimodule. In particular,
the assignment (\ref{Di}) does not extend to a double derivation on $\wh{kQ} $ since its value on a general formal series will not lie in the algebraic tensor product $\wh{kQ}\ot\wh{kQ}$ which admits only finite sums.  Thus we need an alternative definition of double derivations of $\wh{kQ}$  to deal with noncommutative calculus on $\wh{kQ}$.

\begin{remark}
Different from that of $kQ$, the $\wh{kQ}$-bimodule $\Omega_{\wh{kQ}|l}$ is not generated by $d a$ for $a\in Q_1$. For example, take $Q$ to be a quiver with one node and one loop. Then $\wh{kQ}\cong k[[x]]$ and consider a formal series $\sum_{n=0}^\infty a_n x^n$ with generic coefficients. Then
\begin{align*}
d\Big(\sum_{n=0}^\infty a_n x^n\Big)&=dx\sum_{n=0}^\infty a_{n+1} x^n+xdx\sum_{n=0}^\infty a_{n+2} x^n+\ldots
\end{align*}
which can not be expressed as a finite sum of  $f\cdot dx \cdot g$ for some formal series $f,g$. It is an easy exercise to show that this can be done only when $\sum_{n=0}^\infty a_n x^n$ is a geometric series.
\end{remark}

Let $\wh{kQ}\wh{\ot} \wh{kQ}$ be the $k$-module whose elements are formal series of the form
$\sum\nolimits_{u,v} A_{u,v}~ u\ot v$, where $u,v$ runs over all paths. This is nothing but the adic completion of $\wh{kQ}\ot\wh{kQ}$ with respect to the ideal $\wh{\fm}\ot\wh{kQ}+\wh{kQ}\ot \wh{\fm}$. It contains $\wh{kQ}\ot \wh{kQ}$  as a subspace  under the identification
\[(\sum_{u} a'_u~u) \ot (\sum_{v} a''_v~v) \mapsto \sum_{u,v} a'_ua''_v~ u\ot v.\]
There are two obvious bimodule structures on $\wh{kQ} \wh{\ot} \wh{kQ}$, which we call the outer and the inner bimodule structures respectively,  extends those  on the subspace $\wh{kQ} \ot \wh{kQ}$.
Unless otherwise stated, we  view $\wh{kQ}\wh{\ot} \wh{kQ}$ as a  $\wh{kQ}$-bimodule with respect to the outer bimodule structure.   In addition, there are linear maps $\wh{\mu}: \wh{kQ}\wh{\ot} \wh{kQ}\to \wh{kQ}$ and $\wh{\tau}: \wh{kQ}\wh{\ot} \wh{kQ} \to  \wh{kQ}\wh{\ot} \wh{kQ}$ given respectively by
\[
\wh{\mu}(\sum_{u,v} a_{u,v} u\ot v) =  \sum_{w} (\sum_{w=uv} a_{u,v})~w \quad \text{and} \quad \wh{\tau} (\sum_{u,v} a_{u,v} u\ot v) = \sum_{u,v} a_{v,u} u\ot v.
\]
Clearly, $\wh{\mu}$ is a bimodule homomorphism extends $\mu$, and $\wh{\tau}$ extends $\tau$.

We call derivations of $\wh{kQ}$ in the $\wh{kQ}$-bimodule $\wh{kQ}\wh{\ot}\wh{kQ}$ \emph{double $l$-derivations} of $\wh{kQ}$. The inner bimodule structure on $\wh{kQ}\wh{\ot}\wh{kQ}$ naturally yields a bimodule structure on the space 
\[
\wh{\dder}_l(\wh{kQ}):= \der_l(\wh{kQ}, \wh{kQ}\wh{\ot}\wh{kQ}).
\]
For any  $\delta \in \wh{\dder}_l(\wh{kQ})$ and any $f\in \wh{kQ}$, we also write $\delta(f)$ in Sweedler's notation as (\ref{Sweedler}), but one shall bear in mind  that this notation is an infinite sum. Clearly, double derivations of $\wh{kQ}$ are uniquely determined by their values on all $a\in Q_1$. Thus, by abuse of notation,  we have double derivations
\[\frac{\partial~}{\partial a}: \wh{kQ} \to \wh{kQ}\wh{\ot} \wh{kQ}, ~~~~ a\mapsto \delta_{a,b}~e_{s(a)}\ot e_{t(a)},~~\text{for}~a\in Q_1\]
extending the double derivations  $\frac{\partial~}{\partial a}:kQ\to kQ\ot kQ$ constructed in (\ref{Di}). Moreover, every double derivation of $\wh{kQ}$ has a unique representation of the form
\begin{align}\label{representation-doub}
\sum_{a\in Q_1} \sum_{u,v} A_{u,v}^{(a)}~ u* \frac{\partial~}{\partial a} * v,~~~~~A_{u,v}^{(a)}\in k,
\end{align}
where $t(u)=t(a)$ and $s(v)=s(a)$, and $*$ denotes the scalar multiplication of the bimodule structure of $\wh{\dder}_l(\wh{kQ})$. The infinite sum (\ref{representation-doub}) makes sense in the obvious way. Further, note that the map $\wh{\mu}\circ-: \wh{\dder}_l(\wh{kQ}) \to \der_l(\wh{kQ})$ is surjective and  the map \[\wh{\mu} \circ \wh{\tau} \circ-: \wh{\dder}_l(\wh{kQ}) \to \Hom(\wh{kQ},\wh{kQ})\] is a bimodule homomorphism. The image of $\wh{\mu} \circ \wh{\tau} \circ-$ is denoted by $\wh{\cder}_l(\wh{kQ})$ and  its elements are called \emph{cyclic $l$-derivations} of $\wh{kQ}$. By abuse of notation, we have cyclic $l$-derivations
\[
D_{a}:= \wh{\mu} \circ \wh{\tau} \circ \frac{\partial~}{\partial a}: \wh{kQ}\to \wh{kQ}
\]
extending the cyclic derivations  $D_{a}: kQ\to kQ$.  By (\ref{representation-doub}), every cyclic $l$-derivation of $\wh{kQ}$ has a decomposition (not necessary unique) of the form
\begin{align}\label{representation-cyc}
\sum_{a\in Q_1} \sum_{u,v} A_{u,v}^{(a)}~ u\cdot D_a\cdot v,~~~~~A_{u,v}^{(a)}\in k,
\end{align}
where $t(u)=t(a)$ and $s(v)=s(a)$.

Note that all derivations and cyclic derivations of $\wh{kQ}$ are continuous with respect to the $\wh{\fm}$-adic topology on $\wh{kQ}$. Consequently,  $\xi([\wh{kQ},\wh{kQ}]^{cl}) \subseteq [\wh{kQ},\wh{kQ}]^{cl}$ for each  derivation $\xi\in \der_l(\wh{kQ})$, and $D([\wh{kQ},\wh{kQ}]^{cl})=0$ for each cyclic derivation $D\in \wh{\cder}_l(\wh{kQ})$. In addition, note that  \[\wh{\mu}(\delta(\phi)) - \wh{\mu}(\wh{\tau} (\delta(\phi))) \in [\wh{kQ},\wh{kQ}]^{cl}\] for all double derivations $\delta\in \wh{\dder}_l(\wh{kQ})$ and all formal series $\phi\in \wh{kQ}$.
Given  an element  $\Phi$ of  \[\wh{kQ}_{\ol{\cy}}:=\wh{kQ}/[\wh{kQ}, \wh{kQ}]^{cl},\]
the previous discussion implies that all its representatives $\phi$ yields the same commutative diagram
\begin{align}
\xymatrix{
\wh{\dder}_l (\wh{kQ}) \ar@{->>}[r]^-{\wh{\mu}\circ \wh{\tau}\circ-}  \ar@{->>}[d]^-{\wh{\mu}\circ-} &   \wh{\cder}_l (\wh{kQ}) \ar[r]^-{\Phi_*} & \wh{kQ}\ar[d]^{\pi} \\
\der_l (\wh{kQ}) \ar[rr]^-{\Phi_{\#}} & & \wh{kQ}_{\ol{\cy}},
}   \label{derivation+}
\end{align}
where $\pi$ is the projection, $\Phi_{\#}: \xi \mapsto \pi(\xi(\phi))$ and  $\Phi_*: D\mapsto D(\phi)$. Clearly, $\Phi_*$ is  a bimodule homomorphism, so  $\im (\Phi_*)$ is a  two sided ideal of $\wh{kQ}$. Moreover, one has
\[
(\{~ \Phi_*(D_a)~ |~ a\in Q_1~ \}) \subseteq \im(\Phi_*) \subseteq (\{~\Phi_*(D_a) ~|~a\in Q_1~ \})^{cl}.
\]
Note that both equalities hold  when $(\{~ \Phi_*(D_a)~ |~ a\in Q_1~ \})\supseteq \wh{\fm}^r$ for some $r\geq0$.

\begin{definition}\label{superpotential-complete}
Elements of $\wh{kQ}_{\ol{\cy}}:= \wh{kQ}/[\wh{kQ},\wh{kQ}]^{cl}$ are called \emph{potentials} of $\wh{kQ}$. Given a potential $\Phi\in \wh{kQ}_{\ol{\cy}}$,   the two sided ideal
\[
\wh{\rJ}(Q,\Phi):= \im(\Phi_*)
\]
is called the   \emph{Jacobi ideal} of  $\Phi$, where  $\Phi_*: \wh{\cder}_l(\wh{kQ}) \to \wh{kQ}$ is the $\wh{kQ}$-bimodule homomorphism occurs  in  Diagram (\ref{derivation+}). The associative algebra
\[
\wh{\Lm}(Q, \Phi) := \wh{kQ}/\wh{\rJ}(Q,\Phi),
\]
is called the \emph{Jacobi algebra} of   $\Phi$. The smallest integer $r\geq 0$ such that $\Phi\in \pi(\wh{\fm}^r)$ is called the \emph{order} of $\Phi$.
\end{definition}

\begin{remark}We fix a linear order on $Q_1$. Recall that two cycles $u,~v$  are {\em conjugate} if there are paths $w_1,w_2$ such that $u=w_1w_2$ and $v=w_2w_1$. Equivalent classes under this equivalence relation are called \emph{necklaces} or \emph{conjugacy classes}. Also recall that a path $u$ is \emph{lexicographically smaller} than another word $v$ if there exist factorizations $u=waw'$ and $v=wbw''$ with $a<b$. Obviously, this order relation restricts to a total order on  all necklaces. Let us call a cycle \emph{standard}  if it is maximal in its necklace. Then every potential of $kQ$ (resp. $\wh{kQ}$)  has a unique representative which is a finite linear combination (resp. formal linear combination) of  standard cycles. We shall refer to such unique representatives of potentials the \emph{canonical representative}.\end{remark}

%In the sequel, by abuse of notation,  we will frequently write $D_a(\Phi):= \Phi_*(D_a)$ for any potential $\Phi\in \wh{kQ}_{\ol{\cy}}$ and any $a\in Q_1$. So $D_a(\Phi)=D_a(\phi)$ for any representative $\phi$ of $\Phi$.

\begin{lemma}\label{Jacobi-closed}
Fix a potential $\Phi\in \wh{kQ}_{\ol{\cy}}$. Suppose that $k$ is noetherian. Then the canonical map
\begin{equation}\label{canonical-map}
\wh{\Lm}(Q,\Phi) \to \lim_{r\to \infty} \wh{kQ}/(\wh{\rJ}(Q,\Phi)+\wh{\fm}^r)
\end{equation}
is an isomorphism of $l$-algebras. Consequently, $\wh{\rJ}(Q,\Phi)= (\{~ \Phi_*(D_a)~ |~ a\in Q_1~ \})^{cl}$.
\end{lemma}

\begin{proof}
Since $k$ is noetherian, every  inverse system  of finitely generated $k$-modules over the direct set $(\mathbb{N},\leq)$ satisfies the Mittag-Leffler condition (See \cite[Tag 0595]{Stack} for the definition), and moreover all such inverse systems form an abelian category. Then by \cite[Tag, 0598]{Stack}, taking inverse limit is an exact functor from this category to the category of $k$-modules.

Let $\fd:\wh{\dder}_l(\wh{kQ}) \to \wh{kQ}$ be the map given by $\delta\mapsto \Phi_*(\wh{\mu}\circ \wh{\tau}\circ \delta)$. By definition, $\fd$ is an $\wh{kQ}$-bimodule homomorphism and  $\im(\fd) =\wh{\rJ}(Q,\Phi)$. Clearly, the exact sequence
\[\wh{\dder}_l(\wh{kQ})\xrightarrow{\fd} \wh{kQ} \to \wh{\Lm}(Q,\Phi)\to 0\]
induces an exact sequence of inverse system of finitely generated $k$-modules over $(\mathbb{N}, \leq)$ as follows:
\begin{equation}\label{inverse-system}
\{~\wh{\dder}_l(\wh{kQ})/ \wh{\dder}_l(\wh{kQ})_{\geq r}~\}_{r\in \mathbb{N}} \to \{~\wh{kQ}/\wh{\fm}^r~\}_{r\in \mathbb{N}} \to  \{~\wh{kQ}/(\wh{\rJ}(Q,\Phi)+\wh{\fm}^r)~\}_{r\in \mathbb{N}} \to 0,
\end{equation}
where $\wh{\dder}_l(\wh{kQ})_{\geq r}:= \sum_{i=0}^r \wh{\fm}^i * \wh{\dder}_l(\wh{kQ}) * \wh{\fm}^{r-i}$. Clearly, all the morphism are $l$-linear. 
Consider the following commutative diagram
\[
\xymatrix{
\wh{\dder}_l(\wh{kQ}) \ar[r]^-{\fd} \ar[d]^-{\eta_1}& \wh{kQ}\ar[r] \ar[d]^-{\eta_2} & \wh{kQ}/ \wh{\rJ}(Q,\Phi)  \ar[r]\ar[d]^-{\eta_3} &0\\
\lim_{r\to \infty}\wh{\dder}_l(\wh{kQ})/ \wh{\dder}_l(\wh{kQ})_{\geq r} \ar[r] & \lim_{r\to\infty} \wh{kQ}/ \wh{\fm}^r \ar[r] & \lim_{r\to \infty}\wh{kQ}/(\wh{\rJ}(Q,\Phi)+\wh{\fm}^r) \ar[r] & 0,
}
\]
where the bottom is  the limit of the sequence (\ref{inverse-system}) and $\eta_1,\eta_2,\eta_3$ are the canonical maps. By the general result mentioned in the last paragraph, the bottom sequence is also exact. It is easy to check that $\eta_1$ and $\eta_2$ is an isomorphism, so is $\eta_3$. This prove the first statement.

To see the second statement, it suffice to show $\wh{\rJ}(Q,\Phi)$ is closed with respect to the $\wh{\fm}$-adic topology. But we have $\wh{\rJ}(Q,\Phi) = \cap_{r\geq0} (\wh{\rJ}(Q,\Phi)+\wh{\fm}^r) = \wh{\rJ}(Q,\Phi)^{cl}$.
\end{proof}

\begin{remark}
In the above lemma, the statement that $\wh{\rJ}(Q,\Phi)=(\{~ \Phi_*(D_a)~ |~ a\in Q_1~ \})^{cl}$ has been claimed in \cite[Lemma 2.8]{KY11} without a proof; and the statement that the canonical map (\ref{canonical-map}) is an isomorphism is equivalent to say that the Jacobi algebra $\wh{\Lm}(Q,\Phi)$ is pseudocompact, which is actually a special case of the general result \cite[Lemma A.12]{KY11}. Nevertheless, to avoid involve too much, we give here a direct demonstration for reader's convenience.
\end{remark}

One may view potentials of $kQ$ as potentials of $\wh{kQ}$ under the canonical injection $kQ_\cy \to \wh{kQ}_{\ol{\cy}}$. Given a potential $\Phi\in kQ_{\cy}$, we have a compare homomorphism of $l$-algebras
\begin{equation*}\label{compare-map}
\Lm(Q,\Phi) \to \wh{\Lm}(Q, \Phi).
\end{equation*}

\begin{lemma}\label{superpotential-compare}
Fix  a potential $\Phi\in kQ_\cy$ of order $\geq2$. Let  $\mathfrak{a}:= \fm / \rJ(Q,\Phi)$.  Suppose that $k$ is noetherian. Then the compare map $\Lm(Q,\Phi) \to \wh{\Lm}(Q,\Phi)$ factors into a composition as follows
\[
\Lm(Q, \Phi) \to \lim_{r\to \infty} \Lm(Q, \Phi)/\mathfrak{a}^r \xrightarrow{\cong} \wh{\Lm}(Q, \Phi),
\]
where the first map is the canonical one and the second map is an $l$-algebra isomorphism. Consequently,  if $\mathfrak{a}$ is nilpotent, then the compare map $\Lm(Q,\Phi) \to \wh{\Lm}(Q, \Phi)$ is an $l$-algebra isomorphism.
\end{lemma}

\begin{proof}
Let $\mathfrak{b}=\wh{\fm}/ \wh{\rJ}(Q,\Phi)$. It is easy to check that the compare map $\Lm(Q,\Phi)\to \wh{\Lm}(Q,\Phi)$ induces an isomorphism of inverse systems over $(\mathbb{N},\leq)$ of the form
\[
\{~\Lm(Q,\Phi)/\mathfrak{a}^r~\}_{r\in \mathbb{N}} \to \{~\wh{\Lm}(Q,\Phi)/\mathfrak{b}^r~\}_{r\in \mathbb{N}}.
\]
Then we have a commutative diagram of $l$-algebra homomorphisms
\[
\xymatrix{
\Lm(Q,\Phi) \ar[r] \ar[d] & \wh{\Lm}(Q,\Phi) \ar[d] \\
\lim_{r\to \infty} \Lm(Q,\Phi)/\mathfrak{a}^r \ar[r]^{\cong} & \lim_{r\to \infty} \wh{\Lm}(Q,\Phi)/\mathfrak{b}^r,
}
\]
where all maps are the natural one. By Lemma \ref{Jacobi-closed}, the map $\wh{\Lm}(Q,\Phi)\to \lim_{r\to \infty} \wh{\Lm}(Q,\Phi)/\mathfrak{b}^r$ is an isomorphism. Thereof the first statement follows. The second statement is clear.
\end{proof}

\begin{remark}
In general, the compare map $\Lm(Q,\Phi) \to \wh{\Lm}(Q, \Phi)$ is neither injective nor surjective.
It may happens that $\Lm(Q,\Phi)$ is not finitely generated  but $\wh{\Lm}(Q,\Phi)$ is as $k$-modules. Also,  it may happens  that $\Lm(Q,\Phi)$ and $\wh{\Lm}(Q,\Phi)$ are finitely generated $k$-modules of different rank. We give three toy examples below, one for each of the considerations.
\begin{enumerate}
\item[$(1)$] Take $k$ to be a field and $Q$ to be the quiver with one node and two loops, i.e. $kQ\cong k\lg x, y\rg$,  and $\Phi= \pi(\frac{1}{2}x^2 + \frac{1}{3} x^3)$.  Then $\Phi_*(D_{x}) = x+x^2$ and $\Phi_*(D_{y})=0$. In this case,  $\Lm(Q,\Phi)= k\lg x,y\rg/(x+x^2)$ and $\wh{\Lm}(Q,\Phi)= k\lgg x,y\rgg/(x)^{cl} \cong k[[y]]$.  The canonical homomorphism $k\lg x,y\rg/(x+x^2) \to k[[y]]$, which  is given by $x\mapsto 0$ and  $y\mapsto y$, has nonzero kernel contains the class of $x$  and has $k[y]$ as its image.

\item[$(2)$]  Take $k$  and $Q$ to be defined as the previous example, and $\Phi= \pi(\frac{1}{2}x^2 + \frac{1}{3} x^3 + \frac{1}{2}y^2 + \frac{1}{3} y^3)$. Then $\Phi_*(D_{x}) = x+ x^2$ and $\Phi_*(D_{y}) = y+y^2$.  It is easy to check that $\Lm(Q,\Phi) = k\lg x,y\rg/(x+x^2, y+y^2)$  is not finite dimensional  but $\wh{\Lm}(Q,\Phi) = k\lgg x,y\rgg/ (x,y)^{cl}$ is of dimension one.

\item[$(3)$]  Take $k$ to be a field and $Q$ to be the quiver with one node and one loop, i.e.  $kQ=k\lg x\rg$ and $\Phi=\pi(\frac{1}{2}x^2+\frac{1}{3}x^3)$. Again, $\Phi_*(D_{x}) = x(1+x)$. So $\Lm(Q,\Phi)$ is two diemsnional but $\wh{\Lm}(Q,\Phi)$ is one dimensional. Note that $\Lm(Q, \Phi)$ is not local.
\end{enumerate}
\end{remark}

There are some relations between Jacobi algebras of complete path algebras and that of power series algebras. Fix $i\in Q_0$. Let $k[[Q_1^{(ii)}]]$ be the power series algebra generated by $Q_1^{(ii)}$. For any 
$f\in k[[Q_1^{(ii)}]]$, we denote by $f_a$ the formal partial derivative of $f$ with respect to $a\in Q_1^{(ii)}$. The \emph{Jacobi algebra} of
$k[[Q_1^{(ii)}]]$ associated to $f$ is defined to be the commutative algebra
\[
\Lm(k[[Q_1^{(ii)}]], f):=k[[Q_1^{(ii)}]] / (f_a|a\in Q^{(ii)}_1).
\]
Let $\iota^{(i)}: \wh{kQ} \to k[[Q_1^{(ii)}]]$ be the $k$-algebra homomorphism  defined by $e_j\mapsto \delta_{ij}$ for nodes $j\in Q_0$, $a\mapsto a$ for arrows $a\in Q_1^{(ii)}$ and $a\mapsto 0$ for arrows $a\in Q_1\backslash Q_1^{(ii)}$. Clearly, $\iota^{(i)}$ factors as 
\[
\wh{kQ}\xrightarrow{\pi} \wh{kQ}_{\ol{\cy}}\xrightarrow{\ol{\iota}^{(i)}} k[[Q_1^{(ii)}]].
\]

\begin{lemma}\label{abelianization}
Fix  a node $i\in Q_0$ and a  potential $\Phi\in \wh{kQ}_{\ol{\cy}}$. We have
\[
\iota^{(i)}(\Phi_*(D_a))=\Big(\ol{\iota}^{(i)}(\Phi)\Big)_a, \quad a\in Q_1.
\]
Consequently, $\iota^{(i)}$ induces a surjective $k$-algebra homomorphism
$\wh{\Lm}(Q,\Phi)\to \Lm(k[[Q_1^{(ii)}]], \iota^{(i)}(\Phi))$.
\end{lemma}
\begin{proof}
If a cycle $w$ contains an arrow $a\notin Q_1^{(ii)}$, then both sides of the equation (with $\Phi$ replaced by $w$) vanish by the definition of $\iota^{(i)}$.
For an cycle $w=b_1\ldots b_r$ with $b_s\in Q_1^{(ii)}$, one can readily check that
\[
\iota^{(i)}\big(D_a(w)\big)=\sum_{\{s|b_s=a\}}\iota^{(i)}\big(b_{s+1}\ldots b_rb_1\ldots b_{s-1}\big)=\iota^{(i)}(w)_{a}.
\]
Since  the maps $\iota^{(i)}$, $D_a$ and $(-)_{a}$ all  commute with taking formal sums, the result follows.
\end{proof}

%%%%%%%%%%%%%%%%%%%%%%%%%%%%%%%%%%%%%%%%%%%%%%%%%%%%%%%

\subsection{Some technical results on complete path algebras}

\begin{lemma}[Chain rule]\label{chain-rule}
Let $Q$ and $Q'$ be two finite quivers. Let $H: \wh{kQ}\to \wh{kQ'}$ be an algebra homomorphism such that $H(\{ e_i ~ |~ i \in Q_0 \}) \subseteq \{ e_j ~|~ j\in Q_0' \}$.
Let $h_a:= H(a)$ for $a\in Q_1$. Then for any $\phi\in \wh{kQ}$ and $\beta\in Q'_1$ we have
\[
\frac{\partial}{\partial \beta}\big(H(\phi)\big)=\sum_{a\in Q_1} \bigg(\frac{\partial h_a}{\partial \beta}\bigg)^\pp * (H\wh{\ot}H) (\frac{\partial \phi}{\partial a})*  \bigg(\frac{\partial h_a}{\partial \beta}\bigg)^\p
\]
with $*$ means the scalar multiplication with respect to the inner bimodule structure on $\wh{kQ'}\wh{\ot} \wh{kQ'}$ and $H\wh{\ot}H: \wh{kQ}\wh{\ot} \wh{kQ} \to \wh{kQ'}\wh{\ot} \wh{kQ'}$ is given by  $u\ot v\mapsto H(u)\ot H(v)$. Consequently,
\begin{equation}\label{chainrule2}
D_{\beta}(H(\phi)) = \sum_{a\in Q_1}  (\frac{\partial h_a}{\partial \beta})^\pp \cdot  H (D_{a} \phi) \cdot (\frac{\partial h_a}{\partial \beta})^\p.
\end{equation}
Here we used the  Sweedler's notation for the double derivation $\frac{\partial~}{\partial \beta}$ acting on $h_a$. Note that because a double derivation takes value in $\wh{kQ'}\cot\wh{kQ'}$, the Sweedler's notation is an infinite sum.\end{lemma}

\begin{proof}
Suppose $\phi=\sum_w A_ww$, where $w$ runs over all paths of $Q$. Then we  have
\begin{eqnarray*}
\frac{\partial}{\partial \beta}\big(H(\phi)\big) &=& \sum\nolimits_w A_w~ \frac{\partial}{\partial \beta}\big( H(w)\big) \\
           &=& \sum_w A_w~ \sum_{a\in Q_1} (\frac{\partial h_a}{\partial \beta})^\pp * (H\wh{\ot}H)\big(\frac{\partial w}{\partial a}\big) * (\frac{\partial h_a}{\partial \beta})^\p\\
           &=& \sum_{a\in Q_1}  (\frac{\partial h_a}{\partial \beta})^\pp * (H\wh{\ot}H) \big (\sum_{w} A_w~ \frac{\partial w}{\partial a} \big )* (\frac{\partial h_a}{\partial \beta})^\p\\
            &=& \sum_{a\in Q_1}  (\frac{\partial h_a}{\partial \beta})^\pp *  (H\wh{\ot}H) \big ( \frac{\partial \phi}{\partial a}\big) * (\frac{\partial h_a}{\partial \beta})^\p,
\end{eqnarray*}
as required. Since $\wh{\mu}\circ \wh{\tau}\circ (H\wh{\ot}H) = H\circ \wh{\mu} \circ \wh{\tau}$, the desired formula for cyclic derivations follows immediately from the justified formula for double derivations.
\end{proof}

The next result (for free algebras) was stated  in \cite[Proposition 1.5.13]{Ginz} without a proof.

\begin{lemma}[Poincare lemma]\label{Poincare}
Let $Q$ be a finite quiver.
Fix $f_a \in e_{s(a)}\cdot \wh{kQ}\cdot e_{t(a)}$ 
for each $a\in Q_1$.  Suppose $k$ contains a subring which is a field of characteristic $0$. Then
\[
\sum_{a\in Q_1} [a,f_a]=0 \Longleftrightarrow  \exists~ \phi\in \wh{kQ} \text{ such that } f_a=D_a(\phi), ~\text{for all}~ a\in Q_1.
\]
\end{lemma}

\begin{proof}
To prove the if part, it suffices to check the case when $\phi$ is a cycle $w=b_1\ldots b_r$. Then
\begin{align*}
\sum_{a\in Q_1} [a, D_a(w)]&=\sum_{a\in Q_1}\sum_{\{s|b_s=a\}} \big (ab_{s+1}\ldots b_r\ldots b_{s-1}-b_{s+1}\ldots b_r\ldots b_{s-1}a \big )\\
&=\sum_{s=1}^r b_sb_{s+1}\ldots b_r\ldots b_{s-1}-\sum_{s=1}^r b_{s+1}\ldots b_r\ldots b_{s-1} b_s\\
&=0.
\end{align*}

Now we prove the only if part. For a collection of formal series $\{f_a|a\in Q_1\}$, 
we define its antiderivative to be an element $\phi\in \wh{kQ}$ defined by
\[
\phi:=\sum_{a\in Q_1} \sum_{r\geq 1} \frac{1}{r}~a \cdot f_a[r-1]
\] where $f_a[r-1]$ is the sum of paths of degree $r-1$ that occur in $f_a$. It suffices to verify that $D_a(\sum_b bf_b[r-1])=rf_a[r-1]$. Without loss of generality, we may simply assume that $f_a$ are homogeneous of degree $r$ for all $a\in Q_1$. Then by the assumption $\sum_{a\in Q_1} f_a a=\sum_{a\in Q_1} af_a$, $\phi$ is invariant under the cyclic permutations. Therefore $D_a\phi=f_a$ by the definition of $D_a$. 
\end{proof}

\begin{lemma}[Inverse function theorem]\label{inverse}
Suppose  $H$ is an $l$-endomorphism of $\wh{kQ}$ preserving $\wh{\fm}$ that  induces an isomorphism $\wh{\fm}/\wh{\fm}^2 \xrightarrow{\cong} \wh{\fm}/\wh{\fm}^2$.  Then $H$ is invertible.
\end{lemma}

\begin{proof}
By assumption we can choose an $l$-automorphism $T:\wh{kQ} \to \wh{kQ}$ induced by a collection of invertible linear transformations on the spaces of arrows with fixed source and target such that the composition $G:=H\circ T$ satisfying
\[
G(a) \equiv a  \mod \wh{\fm}^2, ~~~~~a\in Q_1.
\]
For a formal series $f$, let us denote $f[r]$ to be the sum of all terms  of degree $r$  that occur in $f$. Then
\[
G(f[r]) \equiv f[r] \mod \wh{\fm}^{r+1}.
\]
For $a\in Q_1$, we let $g_{a,1}=a$ and then inductively set
\[
g_{a,r}:= - G(g_{a,1}+\ldots +g_{a,r-1})[r], ~~~ r\geq 2.
\]
Note that $g_{a,r}$ consists of terms of degree $r$. Let $g_a:=\sum_{r\geq 1} g_{a,r}$. Then for $r\geq 1$  we  have
\[
\begin{array}{llll}
G(g_a) & \equiv&  G(g_{a,1}+\ldots +g_{a,r})  & \mod  \wh{\fm}^{r+1}\\
          & \equiv&   G(g_{a,1}+\ldots +g_{a,r-1}) +G(g_{a,r}) & \mod \wh{\fm}^{r+1} \\
          & \equiv&   a -g_{a,r} + G(g_{a,r}) & \mod \wh{\fm}^{r+1} \\
          & \equiv&   a & \mod  \wh{\fm}^{r+1}.
\end{array}
\]
Here, the third ``$\equiv$'' can be easily obtained by induction of $r$.
Consequently,  $G(g_a) =a$ and therefore $G$ is surjective. Now let us consider the $r$-jet space $J^r:= \wh{kQ}/ \wh{\fm}^{r+1}$. They are all  free $k$-modules with a finite basis. The map $G$ then induces an inverse system of ($k$-)linear maps
\[
G_{r}: J^r\to J^r.
\]
Clearly, all $G_{r}$ are surjective and thereof they are all invertible. This can be seen by consider these maps as square matrices with entries in $k$ and then the determinant trick applies. Since $G$ is just the inverse limit of $G_{r}$, it is invertible. Therefore $H$ is also invertible.
\end{proof}

\begin{lemma}[Nakayama Lemma]\label{Nakayama}
Let $Q$ be a finite quiver. Suppose $N$ is an ideal of $\wh{kQ}$ with $N+ \wh{\fm}^{r} \supseteq \wh{\fm}^{r-1}$ for some $r\geq1$. Then $N\supseteq \wh{\fm}^{r-1}$.
\end{lemma}

\begin{proof}
Note that $\wh{\fm}^{r-1}$ is a finitely generated left $\wh{kQ}$-module. Also note that elements of $1+ \wh{\fm}$ are  all invertible in $\wh{kQ}$, so $\wh{\fm}$ contains in the Jacobson radical of $\wh{F}$. By modularity,
\[
N\cap \wh{\fm}^{r-1}+ \wh{\fm}^{r} = (N+ \wh{\fm}^{r}) \cap \wh{\fm}^{r-1}= \wh{\fm}^{r-1}.
\]
Then by \cite[Lemma 4.22]{Lam}, the Nakayama Lemma, we have $N\cap \wh{\fm}^{r-1} = \wh{\fm}^{r-1}$ and so $N\supseteq  \wh{\fm}^{r-1}$.
%Suppose $M:=\wh{KQ}_{\geq k}/N\cap \wh{KQ}_{\geq k}$ is not zero.  Then $M$ must have a minimal non-empty set $\{f_1,\cdots, f_n\}$ of generators as a left $\wh{KQ}$-module. Since $\mathfrak{n}\cdot M=M$, we have
%\[f_n=\sum_{i=1}^na_i\cdot f_i \]with $a_i\in \mathfrak{n}$. It follows that
%\[f_n=(1-a_n)^{-1} \sum_{i=1}^{n-1} a_i\cdot f_i,\]
%which contradicts the minimality of the set  $\{f_1,\cdots, f_n\}$ of generators for $M$.
\end{proof}

%%%%%%%%%%%%%%%%%%%%%%%%%%%%%%%%%%%%%%%%%%%%%%%%%%%%

\section{Noncommutative Mather-Yau theorem}\label{sec:Proof}

This section is devoted to establish a noncommutative analogue of the well-known Mather-Yau theorem  in the hypersurface singularity theory \cite{MY, BY90, Yau84}.

Throughout, $k$ stands for a commutative ring with unit, $Q$ stands for a finite quiver  and $l=kQ_0$, which is a subalgebra of $\wh{kQ}$.  Let $\wh{\fm} \subseteq \wh{kQ}$ be the ideal generated by arrows.
We continue to use the notations appear in Diagram (\ref{derivation+}). So $\pi: \wh{kQ}\to \wh{kQ}_{\ol{\cy}}$ is the projection map, and for each potential $\Phi\in \wh{kQ}_{\ol{\cy}}$ there are maps $\Phi_*: \wh{\cder}_l(\wh{kQ}) \to \wh{kQ}$ and $\Phi_\#: \der_l(\wh{kQ}) \to \wh{kQ}_{\ol{\cy}}$.

Also, let $\wh{\dder}_l^+ (\wh{kQ})$ be the space of  double derivations of $\wh{kQ}$ that map  $\wh{\fm}$ to $\wh{\fm}  \wh{\ot} \wh{kQ}+ \wh{kQ}\wh{\ot}  \wh{\fm}$, and   $\der_l^+ (\wh{kQ})$ (resp. $\wh{\cder}_l^+ (\wh{kQ})$)  the space of derivations (resp. cyclic derivations) of $\wh{kQ}$ that preserves $\wh{\fm}$. It is easy to check that $\wh{\mu}(\wh{\dder}_l^+(\wh{kQ})) = \der_l^+(\wh{kQ})$ and $\wh{\mu}(\wh{\tau}(\wh{\dder}_l^+(\wh{kQ}))) = \wh{\cder}_l^+(\wh{kQ})$.

\subsection{Right equivalence of potentials}

We denote by $\mathcal{G}:=\Aut_l(\wh{kQ},\wh{\fm})$ the group of $l$-algebra automorphisms of $\wh{kQ}$ that preserves $\wh{\fm}$. It is a subgroup of  $\Aut_l(\wh{kQ})$, the group of all $l$-algebra automorphisms of $\wh{kQ}$.  In the case when $k$ is a field,  $\mathcal{G}=\Aut_l(\wh{F})$. Note that $\mathcal{G}$ acts on $\wh{kQ}_{\ol{\cy}}$ in the obvious way.

\begin{definition}\label{right-equivalent}
For potentials $\Phi,\Psi\in \wh{kQ}_{\ol{\cy}}$, we say $\Phi$ is \emph{(formally) right equivalent} to $\Psi$ and write $\Phi \sim \Psi$, if $\Phi$ and $\Psi$ lie in the same $\cG$-orbit.
\end{definition}

For potentials of  the non complete path algebra $kQ$, one may similarly define the \emph{algebraically right equivalence} in terms of the action of the group of $l$-algebra automorphisms of $kQ$ on $kQ_\cy$. It turns out that two potentials of $kQ$ can be algebraically right equivalent but not formally, and vice versa. This subtle difference can be checked by the following example.

\begin{example}
Let $k=\CC$ and $Q$ the quiver with one node and one loop. Consider the potentials $\Phi: = \frac{1}{2}x^2+\frac{1}{3}x^3$ and $\Psi:= \frac{1}{2}(x-1)^2+\frac{1}{3}(x-1)^3$ of $F$. We have $H(\Phi) =\Psi$ for $H: kQ\to kQ$ the $\CC$-algebra automorphism given by $x\mapsto x-1$, so $\Phi$ and $\Psi$ are algebraically right equivalent. Since  $\Phi$ is not a unit of $\wh{kQ}$ but $\Psi$ is,  $\Phi$ and $\Psi$ are not formally right equivalent.  Now consider the potentials $\Phi':= x^2$ and $\Psi':= x^2+ x^3$ of $kQ$. It is straightforward to show that $\Phi'$ and $\Psi'$ are not algebraically right equivalent. But we have $H'(\Phi') = \Psi'$ for $H': \wh{kQ}\to \wh{kQ}$ the $\CC$-algebra automorphism given by $x\mapsto x\cdot u$, where $u$ is the power series expansion of $(1+x)^{1/2}$ at $0$ (so $u^2=1+x$).
\end{example}

The main concern of this section is to explore relations between the isomorphism  of Jacobi algebras of $\wh{kQ}$ and the right equivalence relation of potentials of $\wh{kQ}$.

The next result was already obtained in \cite[Proposition 3.7]{DWZ} in a more general form. For completeness and reader's convenience, we give a demonstration in our own notations.

\begin{prop}\label{Jacobi-transform}
Let  $\Phi \in \wh{kQ}_{\ol{\cy}}$  and $H\in \mathcal{G}$.  Then
\[
H(\wh{\rJ}(Q,\Phi)) =\wh{\rJ}(Q,  H(\Phi)).
\]
Consequently, $H$ induces an isomorphism of $l$-algebras $\wh{\Lm}(Q, \Phi) \cong \wh{\Lm}(Q,  H(\Phi))$.
\end{prop}

\begin{proof}
By the chain rule (Lemma \ref{chain-rule}),  we have $H(\Phi)_*(D_a) \in H(\wh{\rJ}(Q,\Phi))$ for all $a\in Q_1$. Then by Lemma \ref{Jacobi-closed},  $\wh{\rJ}(Q, H(\Phi)) =\{ ~H(\Phi)_*(D_a) ~|~ a\in Q_1~\}^{cl} \subseteq  H(\wh{\rJ}(Q,\Phi))$.
Symmetrically, we have $\wh{\rJ}(Q,\Phi) \subseteq H^{-1}(\wh{\rJ}(Q,H(\Phi)))$. Thus $H(\wh{\rJ}(Q,\Phi)) = \wh{\rJ}(Q,H(\Phi))$.
\end{proof}

\begin{remark}
Proposition \ref{Jacobi-transform} tells us that two right equivalent potentials of $\wh{kQ}$ have isomorphic Jacobi algebras. The converse is not true, even for $k=\CC$ under the additional assumption that the two isomorphic Jacobi algebras are finite dimensional. For example, take $Q$ to be the quiver with one node and one loop, $\Phi: = \frac{1}{2}x^2+\frac{1}{3}x^3$ and $\Psi:= \frac{1}{2}(x-1)^2+\frac{1}{3}(x-1)^3$. Then $\Phi_*(D_x) = x+x^2$ and $\Psi_*(D_x)=(x-1)+ (x-1)^2$. So  $\wh{\Lm}(Q, \Phi) $ and $\wh{\Lm}(Q, \Psi)$ are both isomorphic to the algebra $\CC$. However, $\Phi$ and $\Psi$ are not right equivalent simply because $\Phi$ is not a unit of $\wh{kQ}$ but $\Psi$ is.
\end{remark}

%%%%%%%%%%%%%%%%%%%%%%%%%%%%%%%%%%%%%%%%%%%%%%%%%%%%%%

\subsection{Statement of the nc Mather-Yau theorem}

Given a potential $\Phi\in \wh{kQ}_{\ol{\cy}}$,  let $\wh{\fm}_\Phi:= \wh{\fm}/ \wh{\rJ}(Q, \Phi)$, which is an ideal of $\wh{\Lm}(Q, \Phi)$. By Lemma \ref{Jacobi-closed}, the $\wh{\fm}_\Phi$-adic topology of  $\wh{\Lm}(Q, \Phi)$ is complete.  Let
\[
\wh{\Lm}(Q, \Phi)_{\ol{\cy}}:=\wh{\Lm}(Q, \Phi)/[\wh{\Lm}(Q, \Phi), \wh{\Lm}(Q, \Phi)]^{cl}.
\]
Clearly, if $\wh{\Lm}(Q, \Phi)$ is a finitely generated $k$-module, then 
\[\wh{\Lm}(Q, \Phi)_{\ol{\cy}}=\wh{\Lm}(Q, \Phi)_\cy=HH_0(\wh{\Lm}(Q, \Phi)).\]
The projection map $\wh{kQ} \to \wh{\Lm}(Q, \Phi)$ induces a  natural map 
\[p_\Phi: \wh{kQ}_{\ol{\cy}} \to \wh{\Lm}(Q, \Phi)_{\ol{\cy}}\]
with kernel $\pi(\wh{\rJ}(Q,\Phi))$. For any $\Theta\in \wh{kQ}_{\ol{\cy}}$, we  write  $[\Theta]= [\Theta]_\Phi:= p_\Phi(\Theta).$

Given a homomorphism $\gamma: A\to B$ of $k$-algebras, we denote by $\gamma_*: A_\cy\to B_\cy$ the induced map. Our main result in this section is the next theorem, which is a noncommutative analogue of the Mather-Yau theorem for isolated hypersurface singularities.

\begin{theorem}[nc Mather-Yau Theorem]\label{ncMY}
Let $Q$ be a finite quiver. Let $\Phi, \Psi\in \wh{\CC Q}_{\ol{\cy}}$ two potentials  of order $\geq 3$ such that the Jacobi algebras $\wh{\Lm}(Q, \Phi)$ and $\wh{\Lm}(Q, \Psi)$ are both finite dimensional. Then the following two statements are equivalent:
\begin{enumerate}
\item[$(1)$] There is an $\CC Q_0$-algebra isomorphism $\gamma: \wh{\Lm}(Q, \Phi)\cong\wh{\Lm}(Q, \Psi)$ so that $\gamma_*([\Phi]_\Phi)=[\Psi]_\Psi$.
\item[$(2)$] $\Phi$ and $\Psi$ are right equivalent.
\end{enumerate}
\end{theorem}

\begin{proof}
We postpone the proof to the end of this section. It needs some technical results which have interest in their own right and  will be presented in the following two subsections. Our discussion is modified from that in \cite{MY}, which deals with isolated hypersurface singularities.
\end{proof}

\begin{remark}
Instead of considering the $\cG$-orbits in $\wh{kQ}_{\ol{\cy}}$, we may rephrase an enhanced version of the nc Mather-Yau theorem for $\cG$-orbits in $\wh{kQ}$. Unfortunately, this cannot hold. Take $Q$ to be the quiver with one node and two arrows. Choose any formal series $\phi \in \wh{\fm}^4$ with finite dimensional Jacobi algebra. Then $\psi:=\phi+ x^2y-yx^2 \not\in \wh{\fm}^4$. Clearly, there is an isomorphism of Jacobi algebras for  $\phi$ and $\psi$. However, $\phi$ and $\psi$ are not in the same $\cG$-orbit because any automorphism of $\wh{kQ}$ should preserve powers of $\wh{\fm}$.
\end{remark}

\begin{definition}\label{quasi-homogeneous}
We call a potential $\Phi \in \wh{kQ}_{\ol{\cy}}$ \emph{quasi-homogeneous} if $[\Phi]_\Phi$ is zero in $\wh{\Lm}(Q, \Phi)_{\ol{\cy}}$, or equivalently,  if  $\Phi$ is contained in $\pi(\wh{\rJ}(Q,\Phi))$.
\end{definition}

\begin{corollary}\label{ncMY-QH}
Let $Q$ be a finite quiver.  Let $\Phi, \Psi\in \wh{\CC Q}_{\ol{\cy}}$ two quasi-homogeneous potentials of order  $\geq 3$ such that the Jacobi algebras $\wh{\Lm}(Q,\Phi)$ and $\wh{\Lm}(Q,\Psi)$ are both finite dimensional. Then the following two statements are equivalent:
\begin{enumerate}
\item[$(1)$] There is an $\CC Q_0$-algebra isomorphism $\wh{\Lm}(Q,\Phi)\cong\wh{\Lm}(Q,\Psi)$.
\item[$(2)$] $\Phi$ and $\Psi$ are right equivalent.
\end{enumerate}
\end{corollary}

\begin{proof}
Since $[\Phi]_\Phi$ and $[\Psi]_\Psi$ are both zero, the result is an immediate consequence of  Theorem \ref{ncMY}.
\end{proof}

Quasi-homogeneous potentials  are abundant. Let us call a potential of $\wh{kQ}$ \emph{weighted-homogeneous} if it has a representative with all terms have the same degree with respect to some choice of positive weight of the arrows of $Q$. For example, take $Q$ to be the quiver with one node and two arrows $x,y$. Then $\Phi=\pi(x^2y-\frac{1}{4}y^4)\in \wh{kQ}_{\ol{\cy}}$ is  weighted-homogeneous. Note that all weighted-homogeneous potentials must lie in $kQ_\cy$ because there are only finitely many paths of the same degree with respect to any choice of weights on arrows.

\begin{lemma}
Weighted-homogeneous potentials of $\wh{kQ}$ are  quasi-homogeneous.
\end{lemma}

\begin{proof}
Suppose that all terms of the potential $\Phi\in \wh{kQ}_{\ol{\cy}}$ have the same degree $d$ with respect to the choice of weight  $|a| =r_a>0$.  Then $\Phi = \pi ( \sum_{a\in Q_1}\frac{r_a}{d} a \cdot \Phi_*(D_a) )$, and so $\Phi$ is quasi-homogeneous.
\end{proof}

\begin{remark}
It is easy to check that the set of quasi-homogeneous potentials of $\wh{kQ}$ are  closed under the action of the group $\mathcal{G}$. So for any weighted-homogeneous $\Phi$, the potential $H(\Phi)$ will be quasi-homogeneous but not weighted-homogeneous for some obvious choices of $H\in \mathcal{G}$. It is a natural question to ask that given a quasi-homogeneous potential $\Phi$ does there exist an element $H\in \cG$ such that $H(\Phi)$ is weighted homogeneous. When $Q$ is a $n$-loop quiver and $k=\CC$, the answer is positive when $\wh{\Lm}(Q,\Phi)$ is finite dimensional (see \cite{HuaZhou}). This result can be viewed as the noncommutative analogue of the famous theorem of Saito \cite{Sait}. 
\end{remark}

Recall from the hypersurface singularity theory that a  power series $f\in k[[x_1,\ldots, x_n]]$ is quasi-homogeneous if $f$ lies in the ideal generated by the derivatives  $f_{x_1},\ldots, f_{x_n}$. By Lemma \ref{abelianization}, the abelianization of a quasi-homogeneous potential of $\wh{kQ}$ (at any node) is quasi-homogeneous. However, the next example tells us that the converse is not true, i.e., a non-quasi-homogenous potential may also have quasi-homogeneous abelianization.

\begin{example}\label{non-quasi-homo}
Let $k=\CC$ and $Q$ the quiver with one node and two loops denoted by $a$ and $b$ respectively. So $\wh{\CC Q}$ is the complete free algebra $\CC\lgg a, b\rgg$ and $\CC Q$ is the free algebra $\CC\langle a, b\rangle$. Consider the potential $$\Phi=a^2b-\sum_{r\geq 4}(-1)^r\frac{b^r}{r}.$$  
It is not hard to check that the abelianization of $\Phi$ is quasi-homogeneous. We proceed to show that $\Phi $ itself is not quasi-homogeneous.
By a direct computation,
\[
\Phi_*(D_a)= ab +ba \quad \text{and} \quad \Phi_*(D_b) = (a^2 +a^2b -b^3)(1+b)^{-1}.
\]
So 
\[
\wh{\Lm}(Q,\Phi) = \frac{\CC \lgg a, b\rgg }{(ab+ba, a^2+a^2b-b^3)^{cl}}.
\]
Consider the algebra
\[
S=  \frac{\CC\lg a,b\rg}{(ab+ba,a^2-b^3+a^2b)}.
\]
A direct computation shows that $a^3=0$ in $S$ and  all ambiguities of the rewriting system
\[\{~ ba \mapsto -ab, \quad b^3\mapsto a^2+a^2b, \quad a^3\mapsto 0 ~\}\]
are resolvable. By the  Diamond Lemma (see \cite[Theorem 1.2]{Berg}), $S$ is nine dimensional with basis
\[
1, ~ a, ~ b, ~ a^2, ~ ab,~ b^2, ~ a^2b, ~ ab^2, ~ a^2b^2.
\]
Moreover, $ab^3= b^6=0$ in $S$. In particular, $S$ is a local algebra. By a similar argument of the proof of Lemma \ref{superpotential-compare}, the canonical morphism  $S\to \wh{\Lm}(Q,\Phi)$ is an isomorphism. By the division algorithm with respect to the above rewriting system, 
\[
\Phi  =  \frac{3}{4} a^2b -\frac{1}{20} a^2b^2 \neq0 \quad \text{in} ~~ \wh{\Lm}(Q,\Phi).
\]
Note that the commutator space $[\wh{\Lm}(Q,\Phi), \wh{\Lm}(Q,\Phi)]$ is spanned by $ab$, $a^2b$ and $ab^2$. So
\[
[\Phi] = -\frac{1}{20} a^2b^2 \neq0 \quad \text{in} ~~ \wh{\Lm}(Q,\Phi)_\cy.
\]
Thus by definition $\Phi$ is not a quasi-homogeneous potential. 
\end{example}

%%%%%%%%%%%%%%%%%%%%%%%%%%%%%%%%%%%%%%%%%%%%%%%%%%%%%%

\subsection{Bootstrapping on Jacobi ideals}

The remaining of this section aims to prove the nc Mather-Yau theorem (Theorem \ref{ncMY}).  This subsection is devoted to establish a bootstrap relation between the Jacobi ideal  $\wh{\rJ}(Q,\Phi)$ and the higher Jacobi ideal $\wh{\fm}\cdot \wh{\rJ}(Q,\Phi)+ \wh{\rJ}(Q,\Phi) \cdot \wh{\fm}$ for potentials $\Phi\in \wh{kQ}_{\ol{\cy}}$. The prototype of this argument in the complex analytic case traces back to \cite{MY}.

For our end, we need  the following  separation lemma for power series rings.

\begin{lemma}[Separation lemma]\label{separation}
Let $\mathcal{P}:= k[[x_1,\ldots, x_n]]$ be the power series ring over  $k$ and let $\mathfrak{a}$ be the ideal generated by $x_1,\ldots, x_n$. Let $f_1,\ldots, f_n$ be a sequence of $n$ elements in $\mathfrak{a}$. Suppose that $k$ is a field and the quotient $k$-algebra $\mathcal{P}/(f_1,\ldots, f_n)$ is finite dimensional over $k$.  Then there exists a homomorphism $\eta: \mathcal{P}\to \ol{k}[[T]]$ of $k$-algebras, where $\ol{k}$ is the algebraic closure of $k$, satisfying that  $\eta(\mathfrak{a}) \subseteq (T)$, $\eta(f_1)\neq 0$ and $\eta(f_i)=0$ for $i>1$.
\end{lemma}
\begin{proof}
Let $I$ be the ideal of $\mathcal{P}$ generated by $f_2,\ldots, f_n$. By an easy (Krull) dimension argument, one obtain that $\dim \mathcal{P}/I = 1$ because $\dim \mathcal{P} =n$ and $\dim \mathcal{P}/(f_1,\ldots, f_n)=0$. Fix a minimal prime ideal $\mathfrak{p}$  over $I$ and  let $A:= \mathcal{P}/\mathfrak{p}$. Clearly, $f_1\not \in \mathfrak{p}$ and  $A$ is a noetherian local domain with residue field $k$ and  of  dimension $1$.  Moreover, $A$ is complete by \cite[Tag 0325]{Stack}. Let $B$ be the integral closure of $A$ in the fraction field of $A$. By standard commutative ring theory, $B$ is a domain of dimension $1$ and some maximal ideal of $B$ contains $\mathfrak{a}$.  By \cite[Tag 032W]{Stack}, $B$ is a finite module over $A$. Then, by \cite[Tag 0325]{Stack} again, $B$ is also a noetherian complete local ring. It follows that the inclusion map $A\to B$ is a finite local homomorphism. So the residue field $\mathbb{K}$ of $B$ is of finite dimensional over $k$ and thereof we may identify $\mathbb{K}$ as a $k$-subalgebra of $\ol{k}$. It is well-known that a Noetherian normal local domain of dimension $1$ is regular. By \cite[Tag 0C0S]{Stack}, $B\cong \mathbb{K}[[T]]$ as $\mathbb{K}$-algebras.  Let $\eta: \mathcal{P}\to \ol{k}[[T]]$ be the composition of the following sequence of homomorphisms of $k$-algebras
\[
\mathcal{P} \xrightarrow{} A \xrightarrow{\subseteq } B \xrightarrow{\cong} \mathbb{K}[[T]] \xrightarrow{\subseteq} \ol{k}[[T]],
\]
where the first one is the projection map. Clearly, $\eta$ satisfies the requirements.
\end{proof}

\begin{prop}[Bootstrapping]\label{Bootstrapping}
Let $\Phi\in\wh{kQ}_{\ol{\cy}}$ be a potential of order $\geq 2$. Suppose that $k$ is a field and the Jacobi algebra $\Lm(\wh{Q},\Phi)$ is finite dimensional. Then
\begin{enumerate}
\item[$(1)$] $\Phi \in \pi(\wh{\rJ}(Q,\Phi))$ (i.e. $\Phi$ is quasi-homogeneous) if and only if $\Phi\in \pi(\wh{\fm}\cdot \wh{\rJ}(Q,\Phi)+\wh{\rJ}(Q,\Phi)\cdot \wh{\fm})$.
\item[$(2)$] For any potential $\Psi\in \wh{F}_{\ol{\cy}}$ of order $\geq 2$ with $\wh{\rJ}(Q,\Psi) = \wh{\rJ}(Q,\Phi)$, it follows that $\Phi-\Psi\in \pi(\wh{\rJ}(Q,\Phi))$ if and only if $\Phi-\Psi\in \pi(\wh{\fm}\cdot\wh{\rJ}(Q,\Phi)+\wh{\rJ}(Q,\Phi)\cdot \wh{\fm})$.
\end{enumerate}
\end{prop}

\begin{proof}
The if part of both statements  are obvious. For the only if part, we prove by contradiction. To simplify the notation, let $\phi_a=\Phi_*(D_a)$ and $\psi_a= \Psi_*(D_a)$ for $a\in Q_1$. Write $Q_1=A\sqcup L$, where $A$ consists of arrows between distinct nodes and $L$ consists of loops. For a power series $f\in \ol{k}[[T]]$, let $o(f)$ denote the order of $f$. It  is defined to be the minimal degree of terms that occurs in $f$. 

For every node $i$, let $\iota^{(i)}: \wh{kQ}\to k[[Q_1^{(ii)}]]$ be the abelianization map constructed before Lemma \ref{abelianization}. For every $a\in Q^{(ii)}_1$, let $f_a = \iota^{(i)}(\phi_a)$. Clearly, $k[[Q_1^{(ii)}]]/ (f_a: a\in Q^{(ii)}_1)$ is finite dimensional over $k$. If the set $Q^{(ii)}_1\neq \emptyset$, then by Lemma \ref{separation},  for any  arrow $b\in Q^{(ii)}_1$ one may choose a local homomorphism $\eta_b^{(i)}: k[[Q_1^{(ii)}]] \to \ol{k}[[T]]$ such that $\eta_b^{(i)}(f_b) \neq 0$ but $\eta_b^{(i)}(f_a) =0$ for $a\neq b$.  Define a $k$-algebra homomorphism $\omega_b^{(i)}: \wh{kQ} \to \ol{k}[[T]]$ by $$\omega_b^{(i)}  = \eta_b^{(i)} \circ \iota^{(i)}.$$   Clearly,  $\omega_b^{(i)}(\wh{\fm}) \subseteq (T)$, $\omega_b^{(i)}(e_j) = \delta_{ij}$, $\omega_b^{(i)}(\phi_b)\neq 0$ and $\omega_b^{(i)} (\phi_a) =0$ for all $a\neq b$. Moreover,  $\omega_b^{(i)}$ factors through $\pi: \wh{kQ}\to \wh{kQ}_{\ol{\cy}}$ by a linear map $\overline{\omega}_b^{(i)}: \wh{kQ}_{\ol{\cy}} \to \ol{k}[[T]]$. By the chain rule,  
\[
\frac{d~ \ol{\omega}_b^{(i)}(\Phi)}{d~T} = \sum_{s(a)=t(a)=i} \omega_b^{(i)}(\phi_a)\cdot \frac{d~ \omega_b^{(i)}(a)}{d~ T} = \omega_b^{(i)}(\phi_b)\cdot \frac{d~ \omega_b^{(i)}(b)}{d~ T}.
\]
Consequently, $o(\omega_b^{(i)}(\phi_b))< o(\ol{\omega}_b^{(i)}(\Phi))$.

To see the only if part of (1), assume that $\Phi\in \pi(\wh{\rJ}(Q,\Phi))$, i.e.
\[
\Phi =\pi(\sum_{a\in Q_1} h_a \cdot \phi_a  )=\pi(\sum_{a\in A} h_a \cdot \phi_a+\sum_{a\in L} h_a \cdot \phi_a ), \quad \text{where }~ h_a \in e_{s(a)} \cdot \wh{kQ}\cdot  e_{t(a)}.
\]
Clearly, $h_a\in \wh{\fm}$ for $a\in A$. Suppose that for some $b \in L$ such that $s(b)=t(b)=i$,
\[
 h_{b}=\sum_w \lambda_{b,w}\cdot w\in e_i\cdot \wh{kQ}\cdot e_i,
\] 
where $w$ runs over cycles based at $i$ and $\lambda_{b,w}\in k$, having $\lambda_{b,e_i}\neq 0$. Since
\[
\ol{\omega}_b^{(i)}(\Phi) = \sum_{a\in Q_1} \omega_b^{(i)}(h_a)\cdot \omega_b^{(i)}(\phi_a) = \omega_b^{(i)}(h_b)\cdot \omega_b^{(i)}(\phi_b)
\]
and $\omega_b^{(i)}(h_b)$ is a unit of $\ol{k}[[T]]$, we then obtain $o(\omega_b^{(i)}(\phi_b))=o(\ol{\omega}_b^{(i)}(\Phi))$, which is absurd.

To see the only if part of (2), let
\[
\Psi-\Phi=\pi(\sum_{a\in Q_1} g_a \cdot \phi_a )=\pi(\sum_{a\in A} g_a \cdot \phi_a+\sum_{a\in L} g_a \cdot \phi_a ), \quad \text{where }~ g_a \in e_{s(a)} \cdot \wh{kQ}\cdot  e_{t(a)}
\]
Again, it is clear that $g_a\in \wh{\fm}$ for $a\in A$. Suppose that for  some $b \in L$ such that $s(b)=t(b)=i$,
\[
 g_{b}=\sum_w \lambda'_{b,w}\cdot w\in e_i\cdot \wh{kQ}\cdot e_i,
\] 
where $w$ runs over cycles based at $i$ and $\lambda'_{b,w}\in k$, having $\lambda'_{b,e_i}\neq 0$. 
Since
\[\ol{\omega}_b^{(i)} (\Psi) = \ol{\omega}_b^{(i)}(\Phi) + \sum_{a\in Q_1}\omega_b^{(i)}(g_a)\cdot \omega_b^{(i)}(\phi_a) = \ol{\omega}_b^{(i)}(\Phi) + \omega_b^{(i)}(g_b)\cdot \omega_b^{(i)}(\phi_b)\] and $\omega_b^{(i)}(g_b)$ is a unit of $\ol{k}[[T]]$, we then obtain
$o(\ol{\omega}_b^{(i)}(\Psi)) = o(\omega_b^{(i)}(\phi_b))$. Note that every ideal $J$ of $\ol{k}[[T]]$ is principal; write $o(J)$ for the order of its generator. Since
\[
\frac{d~ \ol{\omega}_b^{(i)}(\Psi)}{d~T} = \sum_{s(a)=t(a)=i} \omega_b^{(i)}(\psi_a)\cdot \frac{d~ \omega_b^{(i)}(a)}{d~ T},
\]
we obtain $o(\omega_b^{(i)}(\psi_a))< o(\ol{\omega}_b^{(i)}(\Psi))$ for at least one loop $a$ such that $s(a)=t(a)=i$. Hence
\[
o\bigg(\big(\omega_b^{(i)}(\psi_a): a\in Q^{(ii)}_1\big)\bigg)< o(\ol{\omega}_b^{(i)}(\Psi)) = o(\omega_b^{(i)}(\phi_b)) = o\bigg( \big(\omega_b^{(i)}(\phi_a): a\in Q^{(ii)}_1\big) \bigg).
\]
But one has $$\big(\omega_b^{(i)}(\psi_a): a\in Q^{(ii)}_1\big) = \omega_b^{(i)} (\wh{\rJ}(Q,\Psi))  = \omega_b^{(i)} (\wh{\rJ}(Q,\Phi)) =  \big(\omega_b^{(i)}(\phi_a): a\in Q^{(ii)}_1\big),$$
which yields a contradiction.
\end{proof}

%%%%%%%%%%%%%%%%%%%%%%%%%%%%%%%%%%%%%%%%%%%%%%%%

\subsection{Finite determinacy}\label{sec:fd}
Given an integer $r\geq0$, the \emph{$r$-th jet space} of $\wh{kQ}$ is defined to be the quotient $l$-algebra $J^r:=\wh{kQ}/\wh{\fm}^{r+1}$. Clearly, the projection map $\wh{kQ}\to J^r$ induces a canonical surjective map
\[q_r: \wh{kQ}_{\ol{\cy}} \to J^r_\cy:=J^r/[J^r,J^r]\]
with kernel $\pi(\wh{\fm}^{r+1})$.  The image of a potential $\Phi \in \wh{kQ}_{\ol{\cy}}$ under this map is denoted by $\Phi^{(r)}$.  For two potentials $\Psi_1,\Psi_2\in \wh{kQ}_{\ol{\cy}}$,  $\Psi_1^{(r)} = \Psi_2^{(r)}$ if and only if their canonical representatives share the same coefficients for standard cycles of length $\leq r$.

\begin{prop}\label{isomorphism-criterior-0}
Let $\Phi, \Psi\in \wh{kQ}_{\ol{\cy}}$ be potentials  satisfy  that $\Phi^{(r)} = \Psi^{(r)}$ in $J_\cy^r$ for some $r>0$.  Suppose $\wh{\rJ}(Q,\Phi) \supseteq \wh{\fm}^{r-1}$. Then
\[\wh{\rJ}(Q,\Phi)=\wh{\rJ}(Q,\Psi).\] Consequently, $\wh{\Lm}(Q, \Phi) = \wh{\Lm}(Q, \Psi)$ as $l$-algebras.
\end{prop}
\begin{proof}
Note that  $D(\wh{\fm}^s)\subseteq \wh{\fm}^{s-1}$ for any $D\in \wh{\cder}_l (\wh{kQ}) $ and any $s\geq 1$. It follows that
\begin{align*}
\wh{\rJ}(Q,\Psi)+\wh{\fm}^{r}= \wh{\rJ}(Q,\Phi)+ \wh{\fm}^{r} \supseteq \wh{\fm}^{r-1}.
\end{align*}
By Lemma \ref{Nakayama}, we also have $\wh{\rJ}(Q,\Psi) \supseteq \wh{\fm}^{r-1}$ and consequently
\[
\wh{\rJ}(Q,\Psi) = \wh{\rJ}(Q,\Psi)+\wh{\fm}^{r}= \wh{\rJ}(Q,\Phi)+ \wh{\fm}^{r} = \wh{\rJ}(Q,\Phi).
\]
The equality of Jacobi algebras is just by definition.
\end{proof}

Given an integer $r\geq0$, let $\mathcal{G}^r$ be the group of all $l$-algebra automorphisms of $J^r=\wh{kQ}/\wh{\fm}^{r+1}$ preserving $\wh{\fm}/\wh{\fm}^{r+1}$. Clearly, the canonical map $\mathcal{G}\to \mathcal{G}^r$ is surjective. A potential $\Phi\in \wh{kQ}_{\ol{\cy}}$ is called \emph{$r$-determined} (with respect to $\mathcal{G}$)  if $\Phi^{(r)}\in \cG^r\cdot \Psi^{(r)}$ implies $\Phi \sim \Psi $ for all $\Psi\in \wh{kQ}_{\ol{\cy}}$. Clearly, it is  equivalent to the condition that $\Phi^{(r)}=\Psi^{(r)}$ implies $\Phi\sim\Psi$ for all $\Psi \in \wh{kQ}_{\ol{\cy}}$. If $\Phi$ is $r$-determined for some $r\geq 0$ then it is called \emph{finitely determined} (with respect to $\mathcal{G}$).

\begin{remark}
If $\Phi\in \wh{kQ}_{\ol{\cy}}$ is $r$-determined for some integer $r\geq0$ then $\Phi \sim \Phi^{(r)}$.
\end{remark}

This subsection is devoted to prove the following theorem, which can be viewed as a noncommutative  analogue of Mather's infinitesimal criterion \cite[Theorem 1.2]{Wall}.

\begin{theorem}[Finite determinacy] \label{finite-determinacy}
Let $Q$ be a finite quiver and  $\Phi\in \wh{\CC Q}_{\ol{\cy}}$ a potential. If the Jacobi algebra $\wh{\Lm}(Q,\Phi)$ is finite dimensional then $\Phi$ is finitely determined. More precisely, if
$\wh{\rJ}(Q,\Phi) \supseteq \wh{\fm}^r$ for some integer $r\geq 0$ then $\Phi$ is $(r+1)$-determined.
\end{theorem}

The proof of this theorem will be addressed after several auxiliary results.

Let us fix some notations once for all. We denote  $K$  for the algebra of entire functions on the complex plane $\CC$. The base ring $k$ that we need below are $\CC$ and $K$. Let $l=\CC Q_0$ and $\wh{\fm}$ the ideal of $\wh{\CC Q}$ generated by arrows.  Let $L= K Q_0$ and $\wh{\fn}$ the ideal of $\wh{KQ}$ generated by arrows.

We identify $\wh{\CC Q}$ (resp. $\wh{\CC Q}_{\ol{\cy}}$) as a subspace of $\wh{KQ}$ (resp. $\wh{KQ}_{\ol{\cy}}$) in the natural way.  Since $l$-algebra automorphisms of $\wh{\CC Q}$ and $L$-algebra automorphisms of  $\wh{KQ}$ are both uniquely determined by their values on the arrows, one may naturally identify the group $\Aut_l(\wh{\CC Q}) = \Aut_l(\wh{\CC Q}, \wh{\fm})$  as a subgroup of  $\Aut_L(\wh{KQ},\wh{\fn})$. For every $t\in \CC$, let
\[
(-)_t: \wh{KQ} \to \wh{\CC Q}, ~~~~~f\mapsto f_t
\]
be the map given by evaluating coefficients at $t$. Furthermore,  there is a map
 \[
(-)_t: \Aut_L(\wh{KQ},\wh{\fn}) \to \Aut_l(\wh{\CC Q},\wh{\fm}), ~~~~~ H\mapsto H_t: a \mapsto H(a)_t, \quad a\in Q_1.
\]
One may consider $H\in \Aut_L(\wh{KQ},\wh{\fn})$ as an analytic curve $t\mapsto H_t$ of $l$-automorphisms of $\wh{\CC Q}$.

Given a formal series  $f=\sum_{w} a_w(t) w\in \wh{KQ}$, the derivative  $\frac{d\, f}{d\, t}$ is defined to be  the formal series $\sum_{w} a^\p_w(t) w$. It is easy to check that take derivation preserves cyclic equivalence relation on $\wh{KQ}$. Consequently, one may naturally define $\frac{d\, \Phi}{d\, t}$ for any superpotential $\Phi\in \wh{KQ}_{\ol{\cy}}$.

%Recall that $F_R:=R\lg x_1,\ldots,x_n\rg$ embeds into $\wh{F}_R$ as the subspace consisting of algebraic elements. We call an element $\Phi\in\wh{KQ}_{\ol{\cy}}$ algebraic if it is the image of an algebraic element under $\pi$. A double derivation $\delta\in\wh{\dder}_R(\wh{F}_R)$ is called algebraic if $\delta(F)\subset F\ot F$.

%\begin{remark}Note that $R$ is not local. We want to consider a group of ``automorphisms" that ``integrates" $\der_R(\wh{F}_R)$. The group of $R$-algebra automorphisms is not big enough. For example, consider $h:x\mapsto x+a$ for $a\in \CC^\times$. This is not an automorphism of $R\lgg x\rgg$ since $f\circ h$ is not defined for some power series $f$. We need to take the \textcolor{red}{automorphism groupoid}.\end{remark}

\begin{lemma}\label{integrate-derivation}
Fix an element  $f_a\in e_{s(a)}\cdot \wh{\fn}\cdot e_{t(a)}$ for every arrow $a\in Q_1$. There exists an automorphism  $H\in {\rm Aut}_{L} (\wh{KQ}, \mathfrak{n})$ such that
\[
 H_0=\id \quad \text{and} \quad\frac{~d\, H(a)}{d~ t} = -H(f_a)\text{ for } a\in Q_1.
\]
\end{lemma}

\begin{proof}
Write $f_a = \sum_{w} \lambda_{a,w}(t) ~w$,
where $w$ runs over all paths and $\lambda_{a,w}(t) \in K$. Note that  $\lambda_{a,e_i} = 0$ for all arrows $a$ and nodes $i$, and $\lambda_{a,w} =0$ for all arrows $a$ and  paths $w$ that don't have the same source and end.
For each pair of nodes $i$ and $j$, let
\[
M_{ij}(t):=[\lambda_{a,b}(t)]_{a,b\in Q_1^{(ij)}},
\]
which  is an $(Q_1^{(ij)} \times Q_1^{(ij)})$-matrix with entries in $K$.
If desired $H$ exists, write
\[
 H(a) = \sum_w \gamma_{a,w}(t)~ w, \quad a\in Q_1.
\]
Then we would have
\begin{enumerate}
\item[(1)] $\gamma_{a,e_i} (t)= 0$ for all arrows $a$ and nodes $i$, and $\gamma_{a,w} (t)=0$ for all arrows $a$ and  paths $w$ that don't have the same source and end.
\end{enumerate}
Moreover, we would also have:
\[
\sum_{|w|\geq 1}\gamma_{a,w}'(t)~w = - \sum_{|w|\geq1} \lambda_{a,w}(t) ~H(w), \quad a\in Q_1.
\]
Compare the coefficients we then would have for each pair of nodes $i$ and $j$  that 
\begin{enumerate}
\item[(2)]  $(\gamma_{a,b}(0))_{a\in Q_1^{(ij)}}=(\delta_{a,b})_{a\in Q_1^{(ij)}}$ for all arrows from $i$ to $j$, and $(\gamma_{a,w}(0))_{a\in Q_1^{(ij)}}=0$ for all paths $w$  from $i$ to $j$ of length
$\geq 2$.
\item[(3)] For paths $w$  from $i$ to $j$ of length $\geq 1$,
\[
(\gamma_{a,w}'(t))_{a\in Q_1^{(ij)}} = -M_{ij}(t)\cdot (\gamma_{a,w}(t))_{a\in Q_1^{(ij)}}+(F_{a,w}(t))_{a\in Q_1^{(ij)}},
\] where $F_{a,w}(t) \in e_i \cdot \wh{KQ}\cdot e_j$ is a finite linear sum of finite products of elements in
\[
\{ \lambda_{a,u}(t)~ |~ 2 \leq  |u| \leq |w|, ~ s(u)=i, ~ t(u)=j \} \bigcup \{\gamma_{b,v}(t)~|~ b\in Q_1, ~ 1 \leq |v| <  |w| \}.
\]
Here the operation of $(Q_1^{(ij)} \times Q_1^{(ij)})$-matrices on $Q_1^{(ij)}$-tuples is defined in the natural way.
\end{enumerate}
By the stand theory of analytic ODE (see \cite{CLbook}), we can construct all these coefficients $\gamma_{a,w}(t)\in K$ by induction on the length of $w$
from the above three conditions (1), (2) and (3). Then of course the induced $L$-algebra endomorphism $H$ of $\wh{KQ}$ satisfies the ODE, and also
\[
[\gamma_{a,b}(t)]_{a,b\in Q_1^{(ij)}} = e^{N_{ij}(t)}, \quad i,j\in Q_0,
\]
where $N_{ij}(t)$ is the antiderivative of $M_{ij}(t)$ with $N_{ij}(0)$ the identity $(Q_1^{(ij)} \times Q_1^{(ij)})$-matrix.
In particular, $[\gamma_{a,b}(t)]_{a,b\in Q_1^{(ij)}}$
 is  invertible with inverse $e^{-N_{ij}(t)}$ for each pair of nodes $i$ and $j$.  By Lemma \ref{inverse}, $H$ is an $L$-algebra automorphism of $\wh{KQ}$ preserves $\wh{\fn}$.
\end{proof}

\begin{lemma}[Local triviality]\label{thm-localtri}
Let $\Theta\in \wh{KQ}_{\ol{\cy}}$ be a potential. Suppose
$\frac{d \Theta}{d t} \in \Theta_\#(\der_L^+(\wh{KQ}))$. Then
there exists  an automorphism $H\in \Aut_L(\wh{KQ},\wh{\fn})$ such that
\[
 H_0=\id \quad \text{and} \quad H(\Theta)=\Theta_0  \text{ in } \wh{KQ}_{\ol{\cy}}.
\]
\end{lemma}
\begin{proof}
Choose $\xi\in \der_L^+(\wh{KQ})$ such that $\frac{d \Theta}{d t} = \Theta_\#(\xi)$. Then there exists a double $L$-derivation
\[
\delta=\sum_{a\in Q_1}^n \sum_{u,v} \lambda_{u,v}^{(a)}(t)~ u* \frac{\partial~}{\partial a} * v \in \wh{\dder}_L^+(\wh{KQ}),
\]
where $u$ and $v$ runs over paths with $t(u) = t(a)$ and $s(v) =s(a)$ respectively,   such that $\xi=\wh{\mu}\circ \delta$.  In $\wh{KQ}_{\ol{\cy}}$, we have
\begin{eqnarray*}
\Theta_\#(\xi)
&=& \pi\bigg ( \Theta_*\big (\wh{\mu}\circ \wh{\tau} \circ \delta \big ) \bigg )\\
&=& \pi \big ( \sum_{a\in Q_1} \sum_{u,v} \lambda_{u,v}^{(a)}(t)~ u\cdot  \Theta_*(D_a) \cdot v \big )\\
&=& \pi \big ( \sum_{a\in Q_1} \sum_{u,v} \lambda_{u,v}^{(a)}(t)~ vu\cdot  \Theta_*(D_a) \big )\\
&=& \pi \big ( \sum_{a\in Q_1} \xi(a)\cdot  \Theta_*(D_a) \big )
\end{eqnarray*}
Note that $\xi(a)\in e_{s(a)} \cdot \wh{\fn}\cdot e_{t(a)}$. Let  $H\in {\rm Aut}_{L} (\wh{KQ}, \mathfrak{n})$ be chosen as in  Lemma \ref{integrate-derivation}. So
\[
H_0=\id \quad \text{and} \quad\frac{~d\, H(a)}{d~ t} = -H(\xi(a))\text{ for } a\in Q_1.
\]
Then, in $\wh{KQ}_{\ol{\cy}}$  we have
\begin{eqnarray*}
\frac{~~d\, H(\Theta)}{d~ t} &=& H(\frac{~d\, \Theta}{d~ t}) + \pi\bigg(\sum_{a\in Q_1}\frac{~d\, H(a)}{d~ t}\cdot H \big (\Theta_*(D_a)\big) \bigg)  \\
&=&H \bigg ( \Theta_\#(\xi)- \pi\big(\sum_{a\in Q_1} \xi(a) \cdot \Theta_*(D_a) \big )\bigg)\\
&=&0.
\end{eqnarray*}
Thus all coefficients of the canonical representative of $H(\Theta) $ are constants and hence $H(\Theta)= H(\Theta)_0= H_0(\Theta_0) = \Theta_0$.
\end{proof}

\begin{lemma}\label{tangent-compare-1}
For any potential $\Phi\in \wh{\CC Q}_{\ol{\cy}} \subseteq \wh{KQ}_{\ol{\cy}} $, we have
\[
\Phi_* \big (\wh{\cder}_l^+ (\wh{\CC Q}) \big ) \supseteq \wh{\fm}^r \Longleftrightarrow \Phi_*  \big (\wh{\cder}_{L}^+ (\wh{KQ}) \big ) \supseteq \wh{\fn}^r, ~~~ r>0.
\]
\end{lemma}

\begin{proof}
We may identify $\wh{\CC Q}\wh{\ot}_\CC \wh{\CC Q}$ with a subspace of  $\wh{KQ}\wh{\ot}_K\wh{KQ}$ in the obvious way. Since double $l$-derivations of $\wh{\CC Q}$ and double $L$-derivations of  $\wh{KQ}$ are uniquely determined by their values at arrows, we may naturally identify $\wh{\dder}_l(\wh{\CC Q})$ with a subspace of $\wh{\dder}_L(\wh{KQ})$. By the definition of cyclic derivations, the forward implication is clear.

To see the backward implication, fix an arbitrary element $f\in \wh{\fm}^r\subseteq \wh{\fn}^r$.
Choose a double $L$-derivation $\delta\in \wh{\dder}_L^+(\wh{KQ})$ with $\Phi_*(\wh{\mu}_{\wh{KQ}} \circ \wh{\tau}_{\wh{KQ}}\circ\delta) = f$.  It is easy to see that every element $g\in \wh{KQ}\wh{\ot}_K\wh{KQ}$  can be uniquely decomposed into the following form:
\[
g=g_1+ t\cdot g_2,  \quad g_1\in \wh{\CC  Q}\wh{\ot}_\CC \wh{\CC Q}, ~~g_2\in  \wh{KQ}\wh{\ot}_K\wh{KQ}.
\]
Using the decompositions
 \[
\delta(a) = \delta(a)_1+ t\cdot \delta(a)_2, \quad a\in Q_1
\]
as above, we get a decomposition
\[
\delta= \delta_1 +t\delta_2, ~~~\delta_1\in \wh{\dder}_l^+ (\wh{\CC Q}) \subseteq \wh{\dder}_L^+(\wh{KQ}), ~~\delta_2\in \wh{\dder}_L^+(\wh{KQ}).
\]
In fact, $\delta_1$ (resp. $\delta_2$) is given by $\delta_1(a) = \delta(a)_1$ (resp. $\delta_2(a) = \delta(a)_2$).
We have
\[
f = \Phi_*(\wh{\mu}_{\wh{KQ}}\circ \wh{\tau}_{\wh{KQ}}\circ \delta)= \Phi_*(\wh{\mu}_{\wh{\CC Q}}\circ \wh{\tau}_{\wh{\CC Q}}\circ \delta_1) +t \cdot \Phi_*(\wh{\mu}_{\wh{KQ}}\circ \wh{\tau}_{\wh{KQ}}\circ  \delta_2).
\]
It follows that $f= \Phi_*(\wh{\mu}_{\wh{\CC Q}}\circ \wh{\tau}_{\wh{\CC Q}}\circ \delta_1) $ and hence $f \in \Phi_*\big(\wh{\cder}_l^+( \wh{\CC Q}) \big)$.
\end{proof}

\begin{proof}[Proof of Theorem  \ref{finite-determinacy}]
It suffices to show the second statement. Suppose  $\wh{\rJ}(Q,\Phi) \supseteq \wh{\fm}^r$. We proceed to show $\Phi$ is $(r+1)$-determined.  Suppose $\Psi\in \wh{\CC Q}_{\ol{\cy}}$ such that $\Psi^{(r+1)}=\Phi^{(r+1)}$.  Let
\[
\Theta:= \Phi+ t(\Psi-\Phi) \in \wh{KQ}_{\ol{\cy}}.
\]
Clearly, we have
\[
\Phi_* \big (\wh{\cder}_l^+ (\wh{\CC Q}) \big ) \supseteq \wh{\fm}^{r+1}.
\]
Then Lemma \ref{tangent-compare-1} tells us
\[
\Phi_* \big (\wh{\cder}_L^+ (\wh{KQ}) \big )   \supseteq \wh{\fn}^{r+1}.
\]
Since $\Theta$ and $\Phi$ has the same $(r+1)$-jet in $\wh{KQ}_{\ol{\cy}}$, it follows readily that
\[
\Theta_*\big (\wh{\cder}_{L} ^+(\wh{KQ}) \big) +\wh{\fn}^{r+2} = \Phi_* \big (\wh{\cder}_{L} ^+ (\wh{KQ}) \big) +\wh{\fn}^{r+2}\supseteq \wh{\fn}^{r+1}.
\]
Then  Lemma \ref{Nakayama}, the Nakayama lemma,  tells us
\[
\Theta_*\big (\wh{\cder}_{L} ^+(\wh{KQ}) \big) \supseteq \wh{\fn}^{r+1}.
\]
Consequently,
\[
\Theta_{\#} \big (\der_{L} ^+ (\wh{KQ})  \big ) = \pi \big(\Theta_* \big (\wh{\cder}_{L} ^+(\wh{KQ})\big )\big )\supseteq \pi(\wh{\fn}^{r+1}) \ni \Psi-\Phi = \frac{~d\, \Theta}{d~ t}.
\]
Apply Lemma \ref{thm-localtri}, there is an automorphism $H\in \Aut_L(\wh{KQ},\wh{\fn})$ such that $H(\Theta) = \Theta_0 = \Phi$. In particular,
$H_1(\Psi)= H_1(\Theta_1) = H(\Theta)_1 = \Phi $
and so $\Psi$ is right equivalent to $\Phi$.
\end{proof}

\subsection{Proof of the nc Mather-Yau theorem (Theorem \ref{ncMY})}

\begin{proof}[Proof of Theorem \ref{ncMY}]
By Proposition \ref{Jacobi-transform}, it remains to show (1) implies (2).

Let $\gamma: \wh{\Lm}(Q,\Phi)\to \wh{\Lm}(Q,\Psi)$ be an $l$-algebra isomorphism such that $\gamma_*([\Phi]_\Phi)=[\Psi]_\Psi$. We claim that $\gamma$ can be lifted to an $l$-algebra automorphism of $\wh{\CC Q}$.
Denote the image of arrows $a\in Q_1$ in $ \wh{\Lm}(Q,\Phi)$ and $\wh{\Lm}(Q,\Psi)$ both by $\ol{a}$. Clearly, any lifting $h_a\in e_{s(a)} \cdot \CC Q \cdot e_{t(a)}$ of $\gamma(\ol{a})$, $a\in Q_1$, lifts $\gamma$ to an $l$-algebra endomorphism $H: a\mapsto h_a$ of $\wh{\CC Q}$. In other words,  we have a commutative diagram:
\[
\xymatrix{
\wh{\CC Q}\ar[d]\ar[r]^H & \wh{\CC Q}\ar[d]\\
 \wh{\Lm}(Q,\Phi)\ar[r]^\gamma &  \wh{\Lm}(Q,\Psi).
}
\]
Recall that $\wh{\fm}_\Phi \subset \wh{\Lm}(Q,\Phi)$ is defined to be $\wh{\fm}/ \wh{\rJ}(Q, \Phi)$ and similarly for $\wh{\fm}_\Psi\subset \wh{\Lm}(Q,\Psi)$. Because $\Phi$ and $\Psi$ are of order  $\geq3$, there is a canonical isomorphism of $l$-bimodules  $\wh{\fm}/\wh{\fm}^2\cong \wh{\fm}_\Phi/\wh{\fm}_\Phi^2\cong \wh{\fm}_\Psi/\wh{\fm}_\Psi^2$.  Because $\gamma$ induces an isomorphism on $\wh{\fm}_\Phi/\wh{\fm}_\Phi^2\cong \wh{\fm}_\Psi/\wh{\fm}_\Psi^2$, $H$ induces an isomorphism on $\wh{\fm}/\wh{\fm}^2$. By Lemma \ref{inverse}, $H$ is invertible.

By the assumption $\gamma_*([\Phi]_\Phi) = [\Psi]_\Psi$ we have $[H(\Phi)]_\Psi=[\Psi]_\Psi$, and by Lemma \ref{Jacobi-transform} we have
\[
\wh{\rJ}(Q,\Psi) = H(\wh{\rJ}(Q,\Phi)) = \wh{\rJ}(Q,H(\Phi)).
\]
Thus, without lost of generality, we may replace $\Phi$ by $H(\Phi)$ and assume in priori that
\[
\wh{\rJ}(Q,\Phi) = \wh{\rJ}(Q,\Psi) \quad \text{ and }\quad [\Phi]_\Phi=[\Psi]_\Psi.
\]
Let $r$ be the minimal integer so that $\wh{\rJ}(Q,\Phi) \supseteq \wh{\fm}^r$.
By finite determinacy (Theorem \ref{finite-determinacy}), it suffice to show that $\Phi^{(s)}$ and $\Psi^{(s)}$ lie in the same orbit of $\cG^s=\Aut_l(J^s)$ for $s=r+1$. If $\Phi^{(s)}=\Psi^{(s)}$ then there is nothing to proof. So we may assume further that $\Phi^{(s)}\neq \Psi^{(s)}$.

Since $J^s_{\cy}$ is a finite dimensional vector space, it has a natural complex manifold structure. Also,  it is not hard to check that $\mathcal{G}^s$ is a complex Lie group acts analytically on $J^s_\cy$. So the orbit $\mathcal{G}^s\cdot \Xi^{(s)}$ is an immersed submanifold of $J^s_\cy$ for any potential $\Xi\in \wh{\CC Q}_{\ol{\cy}}$. We proceed to calculate $T_{\Xi^{(s)}}(\mathcal{G}^s\cdot \Xi^{(s)})$, the tangent space  of $\mathcal{G}^s\cdot \Xi^{(s)}$  at $\Xi^{(s)}$. Let $\der_l^+(J^s)$ be the space of $l$-derivations of $J^s$ satisfying that  $\delta(\wh{\fm}/ \wh{\fm}^{s+1}) \subseteq \wh{\fm}/ \wh{\fm}^{s+1}$. Clearly, the canonical map $\rho_s: \der_l^+ (\wh{\CC Q}) \to \der_l^+(J^s)$ is surjective. We have a commutative diagram of vector spaces over $\CC$ as following:
\[
\xymatrix{
\wh{\dder}_l^+ (\wh{\CC Q}) \ar@{->>}[r]^-{\wh{\mu}\circ \wh{\tau}\circ-}  \ar@{->>}[d]^-{\wh{\mu}\circ-} &   \wh{\cder}_l^+ (\wh{\CC Q}) \ar[r]^-{\Xi_*} & \wh{\CC Q}\ar[d]^{\pi} \\
\der_l^+ (\wh{\CC Q}) \ar@{->>}[d]^{\rho_s} \ar[rr]^-{\Xi_{\#}} & & \wh{\CC Q}_{\ol{\cy}}  \ar[d]^{q_s} \\ 
\der_l^+(J^s) \ar[rr]^{(\Xi^{(s)})_{\#}} && J^s_\cy,
}  
\]
where  $(\Xi^{(s)})_{\#}$ is constructed in Lemma \ref{cycprop} (1).  Recall that $\der_l^+(J^s)$ is the tangent space of $\mathcal{G}^s$ at the identity map, we have
\[
T_{\Xi^{(s)}}(\mathcal{G}^s\cdot \Xi^{(s)}) = \im ((\Xi^{(s)})_{\#}) =q_s \bigg(\pi \big(\wh{\fm}\cdot \wh{\rJ}(Q,\Xi) + \wh{\rJ}(Q,\Xi)\cdot \wh{\fm} \big)\bigg).
\]

Now consider the complex line $\cL:=\{~ \Theta_t^{(s)}=t\Psi^{(s)}+(1-t)\Phi^{(s)}~|~t\in\CC ~ \}$ contained in $J^s_\cy$. By the assumption that $\wh{\rJ}(Q,\Phi) = \wh{\rJ}(Q,\Psi)$, we have
\[T_{\Psi^{(s)}} (\mathcal{G}^s \cdot \Psi^{(s)}) = T_{\Phi^{(l)}} (\mathcal{G}^s\cdot \Phi^{(s)})=q_s \bigg(\pi \big(\wh{\fm}\cdot \wh{\rJ}(Q,\Phi) + \wh{\rJ}(Q,\Phi) \cdot \wh{\fm} \big)\bigg),\]
as subspaces of $J^s_\cy$.  It follows that for any $t$ the tangent space $T_{\Theta_t^{(s)}} (\mathcal{G}^s\cdot \Theta_t^{(s)})$ is a subspace of  $q_s \bigg(\pi \big(\wh{\fm}\cdot\wh{\rJ}(Q,\Phi) + \wh{\rJ}(Q,\Phi) \cdot \wh{\fm} \big)\bigg)$. Let $\cL_0$ be the subset of $\cL$ consisting of those $\Theta_t^{(s)}$ such that
\[
T_{\Theta_t^{(s)}} (\mathcal{G}\cdot \Theta_t^{(s)}) = q_s \bigg(\pi \big(\wh{\fm}\cdot \wh{\rJ}(Q,\Phi) + \wh{\rJ}(Q,\Phi)\cdot \wh{\fm} \big)\bigg).
\]
Then $\Phi$ and $\Psi$ are  both in $\cL_0$. It remains to show that  $\cL_0$ lies in the orbit  $\cG^s \cdot\Phi^{(s)}$. By a standard lemma in the theory of Lie group (c.f. Lemma 1.1 \cite{Wall}), it suffices to check that
\begin{enumerate}
\item[(1)] The complement $\cL\backslash \cL_0$ is a finite set (so $\cL_0$ is a connected smooth submanifold of $J^s_\cy$);
\item[(2)]  For all $\Theta_t^{(s)}\in \cL_0$, the dimension of   $T_{\Theta_t^{(s)}}(\cG^s \cdot \Theta_t^{(s)})$ are the same;
\item[(3)] For all $\Theta_t^{(s)}\in \cL_0$, the tangent space  $T_{\Theta_t^{(s)}}(\cL_0)$  is contained in $T_{\Theta_t^{(s)}}(\cG^s \cdot \Theta_t^{(s)})$.
\end{enumerate}
Condition (1) holds  because $\cL\backslash \cL_0$ corresponds to the locus of the continuous family
\[\{(\Theta_{t}^{(s)})_{\#}: \der_l^+(J^s) \to J^s_\cy \}_{t\in \CC}\]
 of linear maps between two finite dimensional spaces that have lower rank. Condition $(2)$ follows from the construction of $\cL_0$. Note that the tangent space of $\cL_0$ at each of its point  is spanned by $\Phi^{(s)}-\Psi^{(s)} = q_s(\Phi-\Psi)$ in $J^s_\cy$. By Proposition \ref{Bootstrapping} (2), condition (3) holds if $\Phi-\Psi \in \pi(\wh{\rJ}(Q,\Phi))$,  which is equivalent to the assumption that $[\Phi]_\Phi=[\Psi]_\Psi$.
\end{proof}

%%%%%%%%%%%%%%%%%%%%%%%%%%%%%%%%%%%%%%%%%%%%%%%%%%%%%%

\section{A rigidity theorem on complete Ginzburg dg-algebra}\label{sec:CY-Ginzdg}
First we recall the definition of the \emph{complete Ginzburg dg-algebra}  $\hD(Q,\Phi)$ associated to a finite quiver $Q$ and a potential $\Phi\in \wh{kQ}_{\ol{\cy}}$, where $k$ is a fixed field. 
\begin{definition}(Ginzburg)\label{Ginzalg}
Let $Q$ be a finite quiver and $\Phi$ a potential on $Q$. Let $\ol{Q}$ be the graded quiver with the same vertices as $Q$ and whose arrows are 
\begin{enumerate}
\item[$\bullet$] the arrows of $Q$ (of degree 0);
\item[$\bullet$] an arrow $\theta_a:j\to i$ of degree $-1$ for each arrow $a:i\to j$ of $Q$;
\item[$\bullet$] a loop $t_i:i\to i$ of degree $-2$ for each vertex $i$ of $Q$. 
\end{enumerate}
The (complete) Ginzburg (dg)-algebra $\hD(Q,\Phi)$ is the dg $k$-algebra whose underlying graded algebra is the completion (in the category of graded vector spaces) of the graded path algebra $k\ol{Q}$ with respect to the two-sided ideal generated by the arrows of $\ol{Q}$. Its differential is the unique linear endomorphism homogeneous of degree 1 satisfying the Leibniz rule, and which takes the following values on the arrows of $\ol{Q}$:
\begin{enumerate}
\item[$\bullet$] $\fd e_i=0$ for $i\in Q_0$ where $e_i$ is the idempotent associated to $i$;
\item[$\bullet$] $\fd a=0$ for $a\in Q_1$;
\item[$\bullet$] $\fd(\theta_a)=\Phi_*(D_a)$ for  $a\in Q_1$;
\item[$\bullet$] $\fd(t_i)=e_i(\sum_{a\in Q_1}[a,\theta_a])e_i$ for each $i\in Q_0$. 
\end{enumerate}
It follows from Lemma \ref{Poincare} that $\hD(Q,\Phi)$ is a dg-$l$-algebra with $l=kQ_0$. Taking the topology into consideration, $\hD(Q,\Phi)$ is a pseudocompact dg-$l$-algebra in the sense of \cite[Appendix]{KY11}. 
\end{definition}

The degree zero component of   $\hD(Q,\Phi)$ is exactly the complete path algebra $\wh{kQ}$, which is itself a  pseudocompact algebra.  The $(-1)$- component $\hD^{-1}(Q,\Phi)$ consists of formal series of the form
\[
\sum_{a\in Q_1} \sum_{u,v} A_{u,v}^{(i)}~ u\cdot \theta_a \cdot v,~~~~~A_{u,v}^{(i)}\in k,
\]
where $t(u)=t(a)$ and $s(v)=s(a)$. Note that both $\hD^{-1}(Q,\Phi)$ and $\wh{\dder}_l(\wh{kQ})$ are pseudocompact $\wh{kQ}$-bimodules in the sense of \cite[Appendix]{KY11}.
Moreover, we have a commutative diagram
\begin{equation}\label{Ginzburg-superpotential}
\xymatrix{
\hD^{-1}(Q,\Phi) \ar[rr]^-{\fd} \ar[d]^{\cong} & & \wh{kQ} \ar[d]^-{=} \\
\wh{\dder}_l(\wh{kQ}) \ar[r]^-{\wh{\mu} \circ \wh{\tau} \circ- } & \wh{\cder}_l(\wh{kQ}) \ar[r]^-{\Phi_*} & \wh{kQ}.
}
\end{equation}
The map  $\hD^{-1}(Q,\Phi) \xrightarrow{\cong} \wh{\dder}_l(\wh{kQ})$ appears above is the  $\wh{kQ}$-bimodule isomorphism given by
\[
\sum_{a\in Q_1} \sum_{u,v} A_{u,v}^{(i)}~ u\cdot \theta_a \cdot v \mapsto \sum_{a\in Q_1} \sum_{u,v} A_{u,v}^{(i)}~ u * \frac{\partial}{\partial a} * v.
\]
Recall that the notation $*$ denotes the scalar multiplication of the bimodule structure of $\wh{\dder}_l(\wh{kQ})$ induced from the inner bimodule structure of $\wh{kQ}\wh{\ot}\wh{kQ}$.

\begin{lemma}
Let $Q$ be a finite quiver. For any potential $\Phi\in \wh{kQ}_{\ol{\cy}}$, one has
\[\wh{\Lm}(Q, \Phi) = H^0(\hD(Q,\Phi)).\]
\end{lemma}

\begin{proof}
This is a direct consequence of the commutative diagram (\ref{Ginzburg-superpotential}).
\end{proof}

The main theorem of this section is as follows.

\begin{theorem}\label{rigidity-Ginzburg}
Fix a finite quiver $Q$.  Let $\Phi,\Psi\in \wh{\CC Q}_{\ol{\cy}}$ be two potentials of order $\geq 3$, such that the Jacobi algebras
$\wh{\Lm}(Q,\Phi)$ and $\wh{\Lm}(Q,\Psi)$ are both finite dimensional. Assume there is a $\CC Q_0$-algebra isomorphism  $\gamma: \wh{\Lm}(Q,\Phi)\to\wh{\Lm}(Q,\Psi)$ so that
$\gamma_*([\Phi])=[\Psi]$. Then there exists a dg-$\CC Q_0$-algebra isomorphism
\[
\xymatrix{
\Gamma: \hD(Q,\Phi)\ar[r]^{\cong} &\hD(Q,\Psi)
}
\]
such that $\Gamma(t_i)=t_i$ for all $i\in Q_0$. 
\end{theorem}

\begin{proof}
Let $l=\CC Q_0$. By the nc Mather-Yau theorem (Theorem \ref{ncMY}), there exists a $l$-algebra automorphism $H$ of $\wh{\CC Q}$ such that $H(\Phi)=\Psi$ in $\wh{\CC Q}_{\ol{\cy}}$. Choose such an automorphism $H$  so that  $a\mapsto h_a$, and let $H^{-1}: a\mapsto h_a^{-1}$ be its inverse. {\bf{Warning}:} $h_a^{-1}$ refers to the component of the automorphism $H^{-1}$ that corresponds to $a$, instead of the inverse of $h_a$.
Define a dg-algebra homomorphism $\Gamma: \hD(Q,\Phi)\to \hD(Q,\Psi)$ by
\[
\Gamma: e_i\mapsto e_i,~~~ a\mapsto h_a, ~~~ \theta_a\mapsto \sum_{\beta\in Q_1} H\bigg(\big(\frac{\partial h_\beta^{-1}}{\partial a}\big)^\pp\bigg) \cdot \theta_\beta\cdot H\bigg(\big(\frac{\partial h_\beta^{-1}}{\partial a}\big)^\p\bigg),~~~ t_i\mapsto t_i.
\]
We need to check that $\Gamma$ is compatible with $\fd$. The above assignment defines a morphism of dg-algebras if and only if the following equalities hold:
\begin{eqnarray*}
\Gamma \big(\fd(\theta_a) \big ) &=&\fd \big(\Gamma(\theta_a) \big),  \quad \quad a\in Q_1;\\
\Gamma\big( \fd(t_i) \big)&=&\fd \big( \Gamma(t_i) \big),  \quad \quad i\in Q_0.
\end{eqnarray*}
We verify the first equality (using chain rule).
\begin{align*}
\Gamma \big(\fd(\theta_a) \big )&=\Gamma \big(\Phi_*(D_a)\big)=\Gamma \big( H^{-1}(\Psi)_*(D_a) \big)\\
&=H\bigg(\sum_{\beta \in Q_1}\big(\frac{\partial h_\beta^{-1}}{\partial a}\big)^\pp\cdot H^{-1}\big(\Psi_*(D_\beta)\big) \cdot \big(\frac{\partial h_\beta^{-1}}{\partial a}\big)^\p\bigg)\\
&=\sum_{\beta \in Q_1} H\bigg(\big(\frac{\partial h_\beta^{-1}}{\partial a}\big)^\pp\bigg) \cdot \Psi_*(D_\beta) \cdot H\bigg(\big(\frac{\partial h_\beta^{-1}}{\partial a}\big)^\p\bigg)\\
&=\fd \big(\Gamma(\theta_a) \big)
\end{align*}
By the canonical identification $\hD^{-1}(Q,\Phi) \cong \wh{\dder}_l(\wh{\CC Q})$, to verify the second equality, it suffices to show the following equality holds for any $a$.
\begin{align*}
a *\frac{\partial}{\partial a}-\frac{\partial}{\partial a}* a ~~=~~&\sum_{\beta\in Q_1} h_\beta  *H\bigg(\big(\frac{\partial h_a^{-1}}{\partial \beta}\big)^\pp\bigg)* \frac{\partial}{\partial a} *H\bigg(\big(\frac{\partial h_a^{-1}}{\partial \beta}\big)^\p\bigg) \\  &- H\bigg(\big(\frac{\partial h_a^{-1}}{\partial \beta}\big)^\pp\bigg)* \frac{\partial}{\partial a}*H\bigg(\big(\frac{\partial h_a^{-1}}{\partial \beta}\big)^\p\bigg)* h_\beta.
\end{align*}
It is enough to check the application of both sides on $b\in Q_1$. For $b\neq a$, both sides equal to zero. The $b=a$ case reduces to the equality
\[
e_{s(a)}\ot a-a\ot e_{t(a)}=\sum_{\beta\in Q_1}   H\bigg(\big(\frac{\partial h_a^{-1}}{\partial \beta}\big)^\p\bigg)\ot h_\beta H\bigg(\big(\frac{\partial h_a^{-1}}{\partial \beta}\big)^\pp\bigg) - H\bigg(\big(\frac{\partial h_a^{-1}}{\partial \beta}\big)^\p\bigg)h_\beta \ot H\bigg(\big(\frac{\partial h_a^{-1}}{\partial \beta}\big)^\pp\bigg)
\]
Applying $H^{-1}\wh{\ot} H^{-1}$ to both sides of the equation, it is equivalent to verify the identity
\[
e_{s(a)}\ot h^{-1}_a-h^{-1}_a\ot e_{t(a)}=\sum_{\beta\in Q_1}  \big(\frac{\partial h_a^{-1}}{\partial \beta}\big)^\p\ot \beta\big(\frac{\partial h_a^{-1}}{\partial \beta}\big)^\pp- \big(\frac{\partial h_a^{-1}}{\partial \beta}\big)^\p \beta \ot \big(\frac{\partial h_a^{-1}}{\partial \beta}\big)^\pp.
\]
We claim this holds for arbitrary path $w = a_1\ldots a_r$ with $s(w)=s(a)$ and $t(w)=t(a)$, and therefore holds in general.  Indeed, we have
\[
e_{s(a)}\ot w - w\ot e_{t(a)} =\sum_{\beta\in Q_1}\sum_{a_s=\beta} \big( a_1\ldots a_{s-1} \ot a_{s}  a_{s+1} \ldots a_r -a_1\ldots a_{s-1} a_s \ot a_{s+1} \ldots a_r \big).
\]
Since
\[
\frac{\partial w}{\partial \beta}=\sum_{a_s=\beta}a_1\ldots a_{s-1} \ot a_{s+1} \ldots a_r,
\]
the desired identity follows.

Similarly, we may define a dg-morphism
\[
\Gamma^{-1}: e_i\mapsto e_i, ~~~ b\mapsto h^{-1}_b, ~~~ \theta_b\mapsto \sum_{\alpha\in Q_1} H^{-1}\bigg(\big(\frac{\partial h_\alpha}{\partial b}\big)^\pp\bigg)\cdot \theta_\alpha \cdot H^{-1}\bigg(\big(\frac{\partial h_\alpha}{\partial b}\big)^\p\bigg),~~~ t_i\mapsto t_i.
\]
Apply the canonical identification $\hD^{-1}(Q,\Phi) \cong \wh{\dder}_l(\wh{\CC Q})$ again, to prove $\Gamma^{-1}$ is the inverse of $\Gamma$, it suffices to check the identity
\[
\frac{\partial}{\partial b}=\sum_{\alpha\in Q_1}\sum_{\beta\in Q_1} \big(\frac{\partial h_\alpha^{-1}}{\partial b}\big)^\pp *H^{-1}\bigg(\big(\frac{\partial h_\beta}{\partial b}\big)^\pp\bigg) * \frac{\partial}{\partial \beta} *H^{-1}\bigg(\big(\frac{\partial h_\beta}{\partial b}\big)^\p\bigg)* \big(\frac{\partial h_\alpha^{-1}}{\partial b}\big)^\p
\]
It suffices to check the application of both sides on all $a$ such that $s(a)=s(b)$ and $t(a)=t(b)$. Then we get
\begin{align*}
\sum_{\alpha\in Q_1} H^{-1}\bigg(\big(\frac{\partial h_a}{\partial b}\big)^\p\bigg) \big(\frac{\partial h_\alpha^{-1}}{\partial b}\big)^\p\ot \big(\frac{\partial h_\alpha^{-1}}{\partial b}\big)^\pp H^{-1}\bigg(\big(\frac{\partial h_a}{\partial b}\big)^\pp\bigg) &=\frac{\partial\big(H^{-1}(h_a)\big)}{\partial b}\\
&=\delta_{a,b}e_{s(a)}\ot e_{t(a)}.
\end{align*}
The equality follows by the chain rule.
The desired identity follows.
\end{proof}

\begin{remark}
The condition $\Gamma(t_i)=t_i$ in Theorem \ref{rigidity-Ginzburg} can be interpreted as a volume-preserving condition in noncommutative geometry.
\end{remark}

\begin{remark}
When we finished our proof of Theorem \ref{rigidity-Ginzburg}, we learned that Keller and Yang has already got the fact that every $l$-algebra automorphism  $H$ of $\wh{kQ}$, which transform $\Phi$ to $\Psi$, can be extended to a dg-$l$-algebra isomorphism  $\hD(Q, \Phi) \xrightarrow{\cong} \hD(Q, \Psi)$, see \cite[Lemma 2.9]{KY11}. From this fact and the nc Mather-Yau theorem, one immediately obtain Theorem \ref{rigidity-Ginzburg} as well. However, we retain our proof in full detail for completeness and reader's convenience.
\end{remark}

The correct setup to discuss the complete Ginzburg dg-algebra is to use the language of pseudocompact dg-algebras and derived categories. We now state several definitions and results due to Keller and Van den Bergh. For simplicity, we will omit the definitions of pseudocompact dg-algebras and derived categories. The interested readers can find the details in the Appendix of \cite{KY11} and a generalization in Section 6 of \cite{VdB15}.

Let $l$ be a finite dimensional separable $k$-algebra over a field $k$.
Let $A$ be a pseudocompact dg-$l$-algebra. Denote the pseudocompact derived category of $A$ by $\D_{pc}(A)$. Define the \emph{perfect derived category} $\per_{pc}(A)$ to be the thick subcategory of $\D_{pc}(A)$ generated by the free $A$-module of rank 1. Define the \emph{finite-dimensional derived category} $\D_{fd,pc}(A)$ to be the full subcategory whose objects are the pseudocompact dg-modules $M$ such that
$\Hom_{D_{pc}(A)}(P,M)$ is finite dimensional for each perfect $P$. We say that $A$ is \emph{topologically homologically smooth} if the module $A$ considered as a pseudocompact dg-module over $A^e:=A \wh{\ot} A^{op}$ is quasi-isomorphic to a strictly perfect dg-module (See \cite[Appendix A.11]{KY11} for the definition).

Let $d$ be an integer. For an object $L$ of $\D_{pc}(A^e)$, define $L^\# =\RHom_{A^e}(L,A^e [d])$.
The dg-algebra $A$ is \emph{topologically bimodule $d$-Calabi-Yau} if there is an isomorphism
\[
A \xrightarrow{\cong} A^\#
\]
in $\D_{pc}(A^e)$. In this case, the category $\D_{fd,pc} (A)$
is $d$-Calabi-Yau as a triangulated category. If a  pseudocompact dg-algebra $A$ is topologically homologically smooth and topological $d$-Calabi-Yau as a bimodule, we say $A$ is a \emph{topological  $d$-Calabi-Yau algebra}. In this case, $\D_{fd,pc}(A)$ is a full subcategory of $\per_{pc}(A)$ (Proposition A.14 (c) \cite{KY11}).

Note that  if a pseudocompact dg-algebra $A$ is topologically $d$-Calabi-Yau and has cohomology concentrating in degree $0$ then $H^0(A)$ is also topologically $d$-Calabi-Yau as a pseudocompact algebra (see \cite[Proposition A.14 (e) ]{KY11}). Here, the  topological structure on $H^0(A)$ is inherited from $A$. If $A$ is algebraic (non-topological), then the notion of homologically smooth, bimodule Calabi-Yau and $d$-Calabi-Yau algebra can be defined similarly but with the pseudocompact derived category $\D_{pc}$ replaced by the algebraic derived category.

\begin{theorem}(\cite[Theorem A.17]{KY11})\label{GinzCY}
Complete Ginzburg dg-algebras are topologically $3$-Calabi-Yau.
\end{theorem}

Assume a pseudocompact dg-$l$-algebra $A$ satisfies the following additional conditions: 
\begin{enumerate}
\item[$(1)$] A is topologically $3$-Calabi-Yau,
\item[$(2)$] for each $p>0$, the space $H^p(A)$ vanishes,
\item[$(3)$] the algebra $H^0(A)$ is finite-dimensional over $k$.
\end{enumerate}
The \emph{topological generalized cluster category} of $A$ is defined to be the Verdier quotient 
$$\per_{pc}(A)/\D_{fd,pc}(A).$$ 
For a non-topological dg-$l$-algebra $A$,  which is  (non-topological) $3$-Calabi-Yau and satisfies conditions (2) and (3) above, the \emph{algebraic generalized cluster category} $\per(A)/\D_{fd}(A)$ was first studied by Amiot (see Theorem 2.1 \cite{Am}).

As an immediate consequence of  Theorem \ref{rigidity-Ginzburg}, we have

\begin{corollary}\label{clustercat}
Fix a finite quiver $Q$.  Let $\Phi,\Psi\in \wh{\CC Q}_{\ol{\cy}}$ be two potentials of order $\geq 3$, such that the Jacobi algebras
$\wh{\Lm}(Q,\Phi)$ and $\wh{\Lm}(Q,\Psi)$ are both finite dimensional. Assume there is an $\CC Q$-algebra isomorphism  $\gamma: \wh{\Lm}(Q,\Phi)\to\wh{\Lm}(Q,\Psi)$ so that
$\gamma_*([\Phi])=[\Psi]$. Then the topological generalized cluster categories of $\hD(Q,\Phi)$ and of $\hD(Q,\Psi)$ are triangle equivalent. \hfill $\Box$
\end{corollary}

\begin{remark}
Let $C$ be a contractible rational curve in a smooth quasi-projective Calabi-Yau threefold $Y$. The derived noncommutative deformation of $\cO_C$ is represented by a Ginzburg algebra $\wh{\fD}(Q,\Phi)$, and the underived deformation is represented by the Jacobi algebra $\wh{\Lm}(Q,\Phi)$.  Its generalized cluster category is equivalent to the derived category of singularity for contraction of $Y$. Corollary \ref{clustercat} is used in \cite{HuaKeller} to prove that the Jacobi algebra $\wh{\Lm}(Q,\Phi)$ together with the class $[\Phi]$ classifies all three dimensional flops.
\end{remark}


\begin{thebibliography}{XXXX}
\bibitem{Am} C. Amiot, \emph{Cluster categories for algebras of global dimension 2 and quivers with potential.} (English, French summary)
Ann. Inst. Fourier (Grenoble) 59, no. 6 (2009): 2525--2590.
\bibitem{BY90} M. Benson, S. S.-S. Yau, \emph{Equivalences between isolated hypersurface singularities}, Math. Ann. 287 (1990): 107-134.
\bibitem{Berg} G. Bergman, \emph{The diamond  lemma for ring theory}, Adv.
Math. 29, no. 2 (1979): 178-218.
\bibitem{BCFHR} K. Behrend, I. Ciocan-Fontanine, J. Hwang, M. Rose, \emph{The derived moduli space of stable sheaves}, Algebra Number Theory 8, no. 4 (2014): 781-812.
\bibitem{CLbook} E. A. Coddington, L. Norman, \emph{Theory of ordinary differential equations}, New York, Toronto, London: McGill-Hill Book Company." Inc. XII (1955).
\bibitem{Da} B. Davison, \emph{Superpotential algebras and manifolds.} Advances in Mathematics 231, no. 2 (2012): 879-912.
\bibitem{DWZ} H. Derksen, J. Weyman,  A. Zelevinsky, \emph{Quivers with potentials and their representations I: Mutations}, Selecta Math. 14 (2008): 59-119.
\bibitem{GH85} T. Gaffney, H. Hauser, \emph{Characterizing singularities of varieties and of mappings.} Inventiones mathematicae 81, no. 3 (1985): 427--447.
\bibitem{Ginz} V. Ginzburg, \emph{Calabi-Yau algebras}, preprint, arXiv: 0612139v3 (2007).
\bibitem{GP} G.-M. Greuel, T. H. Pham, \emph{Mather-Yau theorem in positive characteristic}, J. Algebraic Geom. 26 (2017): 347-355.
\bibitem{HuaKeller} Hua, Zheng, and Bernhard Keller. \emph{Cluster categories and rational curves} arXiv preprint arXiv:1810.00749 (2018).
\bibitem{HuaZhou} Hua, Zheng, and Gui-song Zhou. \emph{Quasi-homogeneity of superpotentials} arXiv preprint arXiv:1808.03754 (2018).
\bibitem{KeICM} Keller, Bernhard. \emph{On differential graded categories.} In Proceedings of the International Congress of Mathematicians Madrid, August 22–30, 2006, pp. 151-190. 2007.
\bibitem{KV09} B. Keller, \emph{Deformed Calabi-Yau completions}, J. Reine Angew. Math. 654 (2011): 125-180. With an appendix by Michel Van den Bergh.
\bibitem{KY11} B. Keller, D. Yang, \emph{Derived equivalences from mutations of quivers with potentials}, Adv. Math. 226 (2011): 2118-2168.
\bibitem{Lam} T. Y. Lam, \emph{A first course in noncommutative rings}, Graduate studies in mathematics, Vol. 131, Springer-Verlage (1991).
\bibitem{Ma68} J. N. Mather, \emph{Stability of $C^{\infty}$ mappings, III: Finitely determined map-germs}, Publications Mathématiques de l'Institut des Hautes Études Scientifiques 35, no. 1 (1968): 127-156.
\bibitem{Ma69} J. N. Mather, \emph{Stability of $ C^\infty $ mappings, IV: Classification of stable germs by $ R $-algebras}, Publications Mathématiques de l'IHÉS 37 (1969): 223-248.
\bibitem{MY} J. N. Mather, S. S.-T. Yau, \emph{Classification of isolated hypersurfaces singularitiesby by their moduli algebras}, Invent. Math. 69 (1982): 243-251.
\bibitem{RRS} G.-C. Rota, B. Sagan, P. R. Stein, \emph{A cyclic derivative in noncommutative algebra}, J. Alg. 64 (1980): 54-75.
\bibitem{Sait} K. Saito, \emph{Quasihomogene isolierte Singularit{\"a}ten von Hyperfl{\"a}chen},  Invent. Math. 14 (1971): 123-142.
\bibitem{Stack} The Stacks Project Authors,  \emph{The Stacks project},  http://stacks.math.columbia.edu (2018).
\bibitem{VdB15} M. Van den Bergh,  \emph{Calabi-Yau algebras and superpotentials}, Selecta Mathematica 21 (2015): 555-603.
\bibitem{Wall} C. T. C. Wall, \emph{Finite determinacy of smooth map-germs}, Bull. London Math. Soc. 13 (1981): 481-539.
\bibitem{Yau84} S. S.-T. Yau, \emph{Criterions for right-left equivalence and right equivalence of holomorphic functions with isolated critical points}, Proc. Symp. Pure Math. 41 (1984): 291-297.
\end{thebibliography}
\end{document}